%% file: tgt2-splash-paper.tex
\documentclass[12pt]{article}
\usepackage[pdftex]{graphicx,color}
\usepackage{amssymb}
\usepackage{amsmath}

\usepackage{xr}
\newcommand\K{{\cal K}}

\newcommand\PG{{\rm PG}}
\newcommand\PGL{{\rm PGL}}

\newcommand\GF{{\rm\mbox{GF}}}

\newcommand\D{{\cal D}}

\renewcommand{\P}{\mathcal{P}}

\newcommand\st{:}

\renewcommand{\H}{{\mathcal H}}

\newcommand{\Label}{\label}

\newtheorem{theorem}{Theorem}[section]
\newtheorem{lemma}[theorem]{Lemma}
\newtheorem{corollary}[theorem]{Corollary}

\newtheorem{construction}{Construction}

\newenvironment{proof}{\noindent{\bf Proof}\hspace{0.5em}}
    { \null  \hfill $\square$ \par}

\usepackage[mathscr]{eucal}

\newcommand{\Bpi}{{\mathscr B}}
\newcommand{\bbb}{[\Bpi]}
\newcommand{\vbbb}{{\mathcal V}}

\newcommand{\bbc}{{\mathcal C}}

\newcommand{\bbn}{{\mathcal N}}

\renewcommand{\r}{{q}}

\newcommand{\deltam}{I}

\newcommand{\R}{{\cal R}}
\renewcommand{\S}{{\cal S}}
\newcommand{\C}{{\cal C}}
\newcommand{\N}{{\cal N}}

\newcommand{\si}{\Sigma_\infty}
\newcommand{\li}{\ell_\infty}
\newcommand{\abb}{{\cal A(\cal S)}}
\newcommand{\pbb}{{\cal P(\cal S)}}

\newcommand{\orsps}{order-$\r$-subplanes}
\newcommand{\orsls}{order-$\r$-sublines}
\newcommand{\orsp}{order-$\r$-subplane}
\newcommand{\orsl}{order-$\r$-subline}

\newcommand{\takeaway}{\backslash}
\renewcommand{\st}{\,|\,}

\newcommand{\CapitalTheta}{\Theta}
\newcommand{\plustheta}{\theta}
\newcommand{\minustheta}{\theta^{\mbox{\rm {-}}}\;\!\!\!}

\newcommand{\ST}{{\mathscr S}_T}
\newcommand{\tangentshadow}{shadow}
\newcommand{\tangentshadows}{shadows}

\begin{document}

\title{The tangent splash in $\PG(6,q)$}


\author{S.G. Barwick and Wen-Ai Jackson
\\ School of Mathematics, University of Adelaide\\
Adelaide 5005, Australia
\\
}

\maketitle


Corresponding Author: Dr Susan Barwick, University of Adelaide, Adelaide
5005, Australia. Phone: +61 8 8303 3983, Fax: +61 8 8303 3696, email:
susan.barwick@adelaide.edu.au



\begin{abstract}
Let $\Bpi$ be a subplane of $\PG(2,q^3)$ of order $q$ that is tangent to
$\li$. Then
the tangent splash of $\Bpi$ is defined to be the set of $q^2+1$ points of
$\li$ that lie on a line of $\Bpi$. In the
Bruck-Bose representation of $\PG(2,q^3)$ in $\PG(6,q)$, we investigate the
interaction between the ruled surface corresponding to $\Bpi$ and the planes
corresponding to the tangent splash of $\Bpi$. 
We then give a geometric construction of the unique \orsp\ determined by a
given tangent splash and a
fixed \orsl. 
\end{abstract}

\section{Introduction}

Let $\Bpi$ be a subplane of $\PG(2,q^3)$ of order $q$ that is tangent to
$\li$ in the point $T$. The {\em tangent splash} $\ST$ of $\Bpi$ is defined to be
the set of points of $\li$ that lie on a line of $\Bpi$. So $\ST$ has
$q^2+1$ points, and contains the point $T$, called the {\em centre}. 
Properties of the tangent splash were investigated in \cite{barwtgt1}, and
this investigation is continued here. 
We will work in the Bruck-Bose representation of $\PG(2,q^3)$ in
$\PG(6,q)$, and study the tangent splash in this setting. 
Sections~\ref{sect:BB} and \ref{coord-subplane} comprise the 
relevant background material for this article.

The \orsp\ $\Bpi$ corresponds to a ruled surface in $\PG(6,q)$, and the
tangent splash $\ST$ corresponds to a set of $q^2+1$ planes in
$\PG(6,q)$. In Section~\ref{sect:structure}, we investigate the geometric interaction
between this ruled surface and plane set. Section~\ref{sect:tgtsubspace}
uses this interaction to describe the tangent subspace of a point $P\in\Bpi$ (defined in
\cite{barwtgt1}) in the Bruck-Bose setting.

In \cite{barwtgt1}, the following result is proved.

\begin{theorem}\Label{splashextsubplane}
In $\PG(2,q^3)$, let $\ST$ be a tangent splash of $\li$ and let $\ell$ be an \orsl\ disjoint from
$\li$ lying on a line which meets
$\ST\backslash\{T\}$. Then there is a unique \orsp\ tangent to $\li$ that contains $\ell$
and has tangent splash $\ST$.
\end{theorem}

In Section~\ref{sect:construct} we start with a tangent splash $\ST$ and such an
\orsl\ $\ell$, and show how to  construct this unique \orsp. This
is a geometric construction in the Bruck-Bose representation
in $\PG(6,q)$.

\section[Bruck-Bose representation]{The Bruck-Bose representation of
  $\PG(2,q^3)$ in $\PG(6,q)$}\Label{sect:BB}

This section contains the necessary definitions, notation  and results about the
Bruck-Bose representation of $\PG(2,q^3)$ in $\PG(6,q)$.

\subsection[Bruck-Bose]{The Bruck-Bose representation}\Label{BBintro}

We begin with a brief introduction to  $2$-spreads in $\PG(5,q)$, 
see \cite{hirs91} for more details. 
A 2-{\em spread} of $\PG(5,\r)$ is a set of $\r^3+1$ planes that partition
$\PG(5,\r)$. A 2-{\em regulus} of $\PG(5,\r)$ is a
set of $\r+1$ mutually disjoint planes $\pi_1,\ldots,\pi_{\r+1}$ with
the property that if a line meets three of the planes, then it meets all
${\r+1}$ of them. Three mutually disjoint planes in $\PG(5,\r)$ lie in a unique
2-regulus. A 2-spread $\S$ is {\em regular} if for any three planes in $\S$, the
2-regulus containing them is contained in $\S$. 

The following construction of a regular 2-spread of $\PG(5,\r)$ will be
needed. Embed $\PG(5,\r)$ in $\PG(5,\r^3)$ and let $g$ be a line of
$\PG(5,\r^3)$ disjoint from $\PG(5,\r)$. Let $g^\r$, $g^{\r^2}$ be the
conjugate lines of $g$; both of these are disjoint from $\PG(5,\r)$. Let $P_i$ be
a point on $g$; then the plane $\langle P_i,P_i^\r,P_i^{\r^2}\rangle$ meets
$\PG(5,\r)$ in a plane. As $P_i$ ranges over all the points of  $g$, we get
$\r^3+1$ planes of $\PG(5,\r)$ that partition $\PG(5,\r)$. These planes form a
regular spread $\S$ of $\PG(5,\r)$. The lines $g$, $g^\r$, $g^{\r^2}$ are called the (conjugate
skew) {\em transversal lines} of the spread $\S$. Conversely, given a regular 2-spread
in $\PG(5,\r)$,
there is a unique set of three (conjugate skew) transversal lines in $\PG(5,\r^3)$ that generate
$\S$ in this way.

We now describe the Bruck-Bose representation of a finite
translation plane $\P$ of order $q^3$ with kernel containing $GF(q)$,
an idea which was developed independently by
Andr\'{e}~\cite{andr54} and Bruck and Bose
\cite{bruc64,bruc66}. 
Let $\si$ be a hyperplane of $\PG(6,\r)$ and let $\S$ be a 2-spread
of $\si$. We use the phrase {\em a subspace of $\PG(6,\r)\takeaway\si$} to
  mean a subspace of $\PG(6,\r)$ that is not contained in $\si$.  Consider the following incidence
structure:
the \emph{points} of $\abb$ are the points of $\PG(6,\r)\takeaway\si$; the \emph{lines} of $\abb$ are the 3-spaces of $\PG(6,\r)\takeaway\si$ that contain
  an element of $\S$; and \emph{incidence} in $\abb$ is induced by incidence in
  $\PG(6,\r)$.
Then the incidence structure $\abb$ is an affine plane of order $\r^3$. We
can complete $\abb$ to a projective plane $\pbb$; the points on the line at
infinity $\li$ have a natural correspondence to the elements of the 2-spread $\S$.
The projective plane $\pbb$ is the Desarguesian plane $\PG(2,\r^3)$ if and
only if $\S$ is a regular 2-spread of $\si\cong\PG(5,\r)$ (see
\cite{bruc69}).
For the remainder of this article we work in $\PG(2,\r^3)$, so $\S$ denotes
 a regular 2-spread of $\PG(5,q)$. 

We use the following notation: if $P$ is an affine point of $\PG(3,q^3)$, we also
use $P$ to refer to the corresponding affine point in $\PG(6,q)$. If $T$
is a point of $\li$ in $\PG(2,q^3)$, we use $[T]$ to refer to the spread
element of $\S$ in $\PG(6,q)$ corresponding to $T$.
More generally, if $X$ is a set of points in $\PG(2,q^3)$, we use $[X]$ to
denote the corresponding set in $\PG(6,q)$.

As $\S$ is a regular spread, $\pbb\cong\PG(2,\r^3)$ and 
we can relate the coordinates of $\PG(2,\r^3)$ and $\PG(6,\r)$ as
follows. Let $\tau$ be a primitive element in $\GF(\r^3)$ with primitive
polynomial $$x^3-t_2x^2-t_1x-t_0,$$ where $t_0,t_1,t_2\in\GF(q)$. Then every element $\alpha\in\GF(\r^3)$
can be uniquely written as $\alpha=a_0+a_1\tau+a_2\tau^2$ with
$a_0,a_1,a_2\in\GF(\r)$. Points in $\PG(2,\r^3)$ have homogeneous coordinates
$(x,y,z)$ with $x,y,z\in\GF(\r^3)$. Let the line at infinity $\li$ have
equation $z=0$; so the affine points of $\PG(2,\r^3)$ have coordinates
$(x,y,1)$. Points in $\PG(6,\r)$ have homogeneous coordinates
$(x_0,x_1,x_2,y_0,y_1,y_2,z)$ with $x_0,x_1,x_2,y_0,y_1,y_2,z\in\GF(\r)$.  Let $\si$ have equation $z=0$. 
Let $P=(\alpha,\beta,1)$ be a point of $\PG(2,\r^3)$. We can write $\alpha=a_0+a_1\tau+a_2\tau^2$ and
$\beta=b_0+b_1\tau+b_2\tau^2$ with $a_0,a_1,a_2,b_0,b_1,b_2\in\GF(\r)$. Then
the map 
\begin{eqnarray*}
\sigma\colon \PG(2,\r^3)\takeaway\li&\rightarrow&\PG(6,\r)\takeaway\Sigma_\infty\\
(\alpha,\beta,1)&\mapsto&(a_0,a_1,a_2,b_0,b_1,b_2,1)
\end{eqnarray*}
 is the Bruck-Bose map. 
More generally, if $z\in\GF(q)$, then we can generalise the map to
$\sigma(\alpha,\beta,z)=(a_0,a_1,a_2,b_0,b_1,b_2,z)$. 
Note that if $z=0$, then $T=(\alpha,\beta,0)$ is a point of $\li$, and
$\sigma(\alpha,\beta,0)$ is a single point in the spread element $[T]$
corresponding to $T$. 

In $\cite{barw12}$, the coordinates of the transversals of the regular spread $\S$ are
calculated.
\begin{lemma}\Label{transversaleqn}
The line $g$ of $\PG(6,q^3)$ joining the points
$(t_1+t_2\tau-\tau^2,t_2-\tau,-1,0,0,0,0)$ and
$(0,0,0,t_1+t_2\tau-\tau^2,t_2-\tau,-1,0)$ is one of the
transversals of the regular spread.
\end{lemma}


\subsection[Tangent subplanes]{Subplanes and sublines in the Bruck-Bose representation}

An \emph{\orsp} of $\PG(2,\r^3)$ is a subplane of
$\PG(2,q^3)$ of order $q$. An \emph{\orsl} of
$\PG(2,q^3)$ is a line of an \orsp\ of $\PG(2,q^3)$. Note that an \orsp\ is
the image of the subplane $\PG(2,q)$ under the collineation group
$\PGL(3,q^3)$, and an \orsl\ is the image of the subline $\PG(1,q)$ under $\PGL(3,q^3)$.
In \cite{barw12}, the authors determine the representation of \orsps\  and
\orsls\ of $\PG(2,q^3)$ in the Bruck-Bose representation in $\PG(6,q)$, 
we quote the results we need here.
 We first introduce some notation to simplify the statements. A {\em special conic} $\C$ is a conic in a spread element, such
that when we extend $\C$ to $\PG(6,q^3)$, it meets the transversals of
the regular spread $\S$. Similarly, a {\em special twisted cubic} $\N$ is a
twisted cubic in a 3-space of $\PG(6,q)\backslash\si$ about a spread
element,  such that when we extend $\N$ to $\PG(6,q^3)$, it meets the
transversals of $\S$. Note that a special twisted cubic has no
points in $\si$.

\begin{theorem}{\rm \cite{barw12}}\Label{sublinesinBB}
Let $b$ be an \orsl\ of $\PG(2,q^3)$. 
\begin{enumerate}
\item\Label{subline-secant-linfty} If $b\subset\li$, then in $\PG(6,q)$, $b$ corresponds to a 
$2$-regulus of $\S$. Conversely every $2$-regulus of $\S$ corresponds to
an \orsl\ of $\li$.
\item If $b$ meets $\li$ in a point, then in $\PG(6,q)$, $b$  corresponds to
  a line of $\PG(6,q)\backslash\si$. Conversely every line of
  $\PG(6,q)\backslash\si$ corresponds to an \orsl\ of $\PG(2,q^3)$ tangent
  to $\li$.
\item\Label{FFA-orsl}
If $b$ is disjoint from $\li$, then in
$\PG(6,q)$, $b$ corresponds to a special twisted cubic. Further, a
twisted cubic $\N$ of $\PG(6,q)$ corresponds to an \orsl\ of
$\PG(2,q^3)$ if and only if $\N$ is special.
\end{enumerate}
\end{theorem}

\begin{theorem}{\rm \cite{barw12}}\Label{FFA-subplane-tangent}
Let $\Bpi$ be an \orsp\ of $\PG(2,\r^3)$.
\begin{enumerate}
\item If $\Bpi$ is secant to $\li$, then
  in $\PG(6,q)$,
  $\Bpi$ corresponds to a plane of $\PG(6,q)$ that meets $q+1$ spread
  elements. Conversely, any plane of $\PG(6,q)$ that meets $q+1$ spread
  elements corresponds to an \orsp\ of $\PG(2,q^3)$ secant to $\li$. 
\item  Suppose  $\Bpi$  is tangent to $\li$ in the
  point $T$. 
Then $\Bpi$ determines a set $\bbb$  of points in $\PG(6,\r)$ (where the
affine points of $\Bpi$ correspond to the affine points of $\bbb$) such that:
\begin{enumerate}
\item[{\rm (a)}]  $\bbb$ is a ruled surface with conic directrix $\bbc$ contained in
  the plane $[T]\in\S$, and twisted cubic directrix $\bbn$ contained in a
  3-space $\Sigma$ that meets $\si$ in a spread element (distinct from
  $[T]$). The points of $\bbb$ lie on $\r+1$ pairwise disjoint generator lines joining $\bbc$ to $\bbn$.
\item[{\rm (b)}]  The $\r+1$ generator lines of $\bbb$ joining $\bbc$
  to $\bbn$ are determined by a projectivity from $\bbc$ to $\bbn$.
\item[{\rm (c)}]  When we extend $\bbb$
  to $\PG(6,\r^3)$, it contains the conjugate transversal
  lines $g,g^\r,g^{\r^2}$ of the regular spread $\S$. So $\C$ and $\N$ are special.
\item[{\rm (d)}]  $\bbb$ is the intersection of nine quadrics in $\PG(6,\r)$.
\end{enumerate}
\end{enumerate}
\end{theorem}

The converse of this second correspondence is also true.

\begin{theorem}{\rm \cite{barw13}}\Label{subplane-tangent-iff}
Let $\bbb$ be a ruled surface of $\PG(6,q)$ defined by a projectivity from a conic
directrix 
$\C$ to a twisted cubic directrix $\N$. 
Then $\bbb$ corresponds  to an \orsp\ of $\PG(2,q^3)$ if and only if
$\C$ is a special conic in
a spread element $\pi$, $\N$ is a special
twisted cubic in a 3-space about a spread element distinct from
$\pi$, and
in the cubic extension $\PG(6,\r^3)$ of $\PG(6,\r)$, $\bbb$ contains the transversals of the regular
spread $\S$. 
\end{theorem}

\subsection[Collineations]{The collineation group in the Bruck-Bose
  representation}

Consider a homography $\alpha\in\PGL(3,\r^3)$ acting on $\PG(2,\r^3)$ that
fixes $\li$ as a set of points. There is a corresponding homography
$[\alpha]\in\PGL(7,\r)$ acting on the Bruck Bose representation of
$\PG(2,\r^3)$ as $\PG(6,\r)$. Note that $[\alpha]$ fixes the hyperplane
$\si$ at infinity of $\PG(6,\r)$, and permutes the elements of the
spread $\S$ in  $\si$.  Consequently subgroups of $\PGL(3,\r^3)$ fixing $\li$ correspond to subgroups of $\PGL(7,\r)$ fixing $\si$ and permuting the spread elements in $\si$.
We will need some properties of these collineations groups, which we
present here.

We first consider the simpler  Bruck-Bose representation of $\PG(1,\r^3)$
in $\PG(3,\r)$. A homography $\alpha$ of $\PG(1,\r^3)$ is an element of  $\PGL(2,\r^3)$ and
can be represented by a $2\times 2$ matrix
$$A=\begin{pmatrix}a&b\\c&d\end{pmatrix},\quad a,b,c,d\in\GF(\r^3),$$ where
$|A|\neq0$. If $\alpha$ fixes the point $(1,0)$, then we can set $c=0,\
d=1$. 
To convert this to a homography $[\alpha]$ in $\PGL(4,q)$ that acts on $\PG(3,q)$,
write $a=a_0+a_1\tau+a_2\tau^2$,
 $b=b_0+b_1\tau+b_2\tau^2$, for $a_i,b_i\in\GF(q)$, and let
 $\sigma(a)=(a_0,a_1,a_2)$, $\sigma(b)=(b_0,b_1,b_2)$.
Let 
\begin{equation}\label{matrixT}
T=\begin{pmatrix}0&0&t_0\\1&0&t_1\\0&1&t_2\end{pmatrix},
\end{equation}
then $[\alpha]$ has matrix 
\renewcommand{\arraystretch}{1.2}
\[
[A]=\left(\begin{array}{ccc|c}
\sigma(a)^t&T\sigma(a)^t&T^2\sigma(a)^t&\sigma(b)^t\\
\hline
0&0&0&1\end{array}\right).
\]
\renewcommand{\arraystretch}{1}

 Note that $[A]$ fixes the plane $\pi_\infty$ of $\PG(3,q)$ with homogeneous coordinates
 $[0,0,0,1]$. Further, the group
 $\langle[T]\rangle$ is a Singer group acting on the points of $\pi_\infty$
 that is regular on the points, and on the lines of $\pi_\infty$ (see
 \cite[Chapter 4]{rey} for more information on Singer cycles).

As a homography of $\PG(1,q^3)$ that fixes the infinite point $(1,0)$ gives rise to a
homography of $\PG(3,q)$ that fixes the infinite plane $[0,0,0,1]$, the
following result is immediate.

\begin{theorem}\Label{PGL2r3}
Let $E$ be the subgroup of $\PGL(2,\r^3)$ fixing the infinite point $(1,0)$
of $\PG(1,q^3)$.
Consider the Bruck-Bose representation of $\PG(1,q^3)$ in $\PG(3,q)$ with
plane at infinity $\pi_\infty=[0,0,0,1]$.
Let $[E]$ be the subgroup of $\PGL(4,\r)$ corresponding to $E$, so $[E]$
acts on $\PG(3,q)$ and  fixes the plane $\pi_\infty$. Then
\begin{enumerate}
\item $|[E]|=\r^3(\r^3-1)$.
\item $E$ is transitive on the affine points of $\PG(1,\r^3)$, and so $[E]$ is transitive on the affine points of $\PG(3,\r)$.
\item $E$ is transitive on the $\r^3(\r^3-1)$ \orsls\ of $\PG(1,\r^3)$ not
  containing $(1,0)$, and so $[E]$ is transitive on the special twisted cubics of $\PG(3,q)$.
\end{enumerate}
\end{theorem}

We now consider the Bruck-Bose representation of $\PG(2,\r^3)$
in $\PG(6,q)$.
Let $\alpha\in\PGL(3,q^3)$ be a homography of $\PG(2,\r^3)$ fixing the line $\li$ of
$\PG(2,\r^3)$.  Then $\alpha$ has matrix
$$A=\begin{pmatrix}a&b&c\\d&e&f\\0&0&1\end{pmatrix}, \quad
a,b,c,d,e,f\in\GF(q^3),$$ with $|A|\neq0$. The corresponding homography $[\alpha]\in\PGL(7,\r)$ of $\PG(6,\r)$ has a $7\times7$ matrix
\renewcommand{\arraystretch}{1.2}
\[
[A]=\left(\begin{array}{ccc|ccc|c}
\sigma(a)^t&T\sigma(a)^t&T^2\sigma(a)^t&\sigma(b)^t&T\sigma(b)^t&T^2\sigma(b)^t&\sigma(c)^t\\
\hline
\sigma(d)^t&T\sigma(d)^t&T^2\sigma(d)^t&\sigma(e)^t&T\sigma(e)^t&T^2\sigma(e)^t&\sigma(f)^t\\
\hline
0&0&0&0&0&0&1\end{array}\right)
\]
with $T$ as in (\ref{matrixT}). It is straightforward to prove the following
result that $T$ gives rise to a Singer cycle in each spread element in $\PG(6,q)$. 

\renewcommand{\arraystretch}{1}
\begin{theorem}\Label{PGL7r}
Consider  the homography $\CapitalTheta\in\PGL(7,\r)$ with $7\times 7$ matrix
$$M=\begin{pmatrix}T&0&0\\0&T&0\\0&0&1\end{pmatrix}, \quad{\rm where}\ \ 
T=\begin{pmatrix}0&0&t_0\\1&0&t_1\\0&1&t_2\end{pmatrix}.$$
Then in $\PG(6,q)$, $\CapitalTheta$ fixes each plane of the spread $\S$, and $\langle\CapitalTheta\rangle$
acts regularly on the set of points, and on the set of lines, of each spread element.
\end{theorem} 

\section{Background results on tangent splashes}\Label{coord-subplane}

In \cite{barwtgt1}, a number of group theoretic properties of tangent
splashes are proved, the most useful is the following.

\begin{theorem} \Label{transitiveontangentsubplanes}
The subgroup of $\PGL(3,q^3)$ acting on $\PG(2,q^3)$ and fixing the line
$\li$  is transitive on all the \orsps\  tangent to $\li$, and transitive on the tangent splashes of $\li$.
\end{theorem}

This means that if we want to prove a result about tangent \orsps\ 
or tangent splashes,
then we can without loss of generality prove it for a particular tangent
subplane or tangent splash. We use the following tangent \orsp\ $\Bpi$ that is coordinatised
in full detail in \cite[Section~\ref{TS-coord-subplane}]{barwtgt1}. 
The labelling we use for the points and lines of $\Bpi$ are illustrated in Figure~\ref{fig-subscripts},
and their coordinates are given in Table~\ref{table-coords}. The tangent
splash of $\Bpi$ is $$\ST=\{(a+b\tau,1,0)\st a,b\in\GF(q)\}\cup\{(1,0,0)\}.$$

\begin{figure}[h]
\centering
\input{tangentsubplane3.pdftex_t}
\caption{Tangent subplane notation}\label{fig-subscripts}
\end{figure}
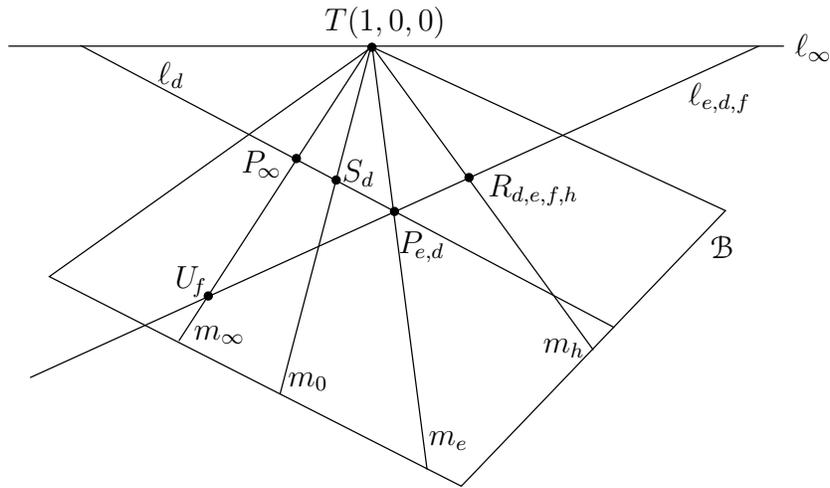

\begin{table}[ht]
\caption{Coordinates of points in the \orsp\ $\Bpi$, ($e,d,f,h\in\GF(q)$)}\label{table-coords}
\begin{center}
\begin{tabular}{|c|c|c|}
\hline
Notation&Coordinates&Description\\
\hline
$T$&$(1,0,0)$&$\Bpi\cap\li$\\
$P_\infty$&$(1,1,1)$&\\
$m_e$&$[0,e+\tau,-e]$&lines of $\Bpi$ through $T$\\
$m_\infty$&$[0,1,-1]$&$TP_\infty$\\
$S_d$&$(d,0,1)$&points of $\Bpi$ on $m_0$\\
$U_f$&$(1+f\tau,1,1)$&points of $\Bpi$ on $m_\infty$\\
$\ell_d$&$[1,d-1,-d]$&$P_\infty S_d$\\
$P_{e,d}$&$(e+d\tau,e,e+\tau)$&$m_e\cap \ell_d$\\
$\ell_{e,d,f}$&$[-1,ef-d+1+f\tau,d-ef]$&$P_{e,d}U_f$\\
$R_{e,d,f,h}$&$(h+(fh-fe+d)\tau,h,h+\tau)$& $\ell_{e,d,f}\cap m_h$\\
$R_{e,d,f,\infty}$&$(1+f\tau,1,1)$&$\ell_{e,d,f}\cap m_\infty=U_f$\\
\hline
\end{tabular}
\end{center}
\end{table}

If we look 
 at the points of a  tangent splash $\ST$ in the
Bruck-Bose representation in $\PG(6,q)$, we have a 
set $[\ST]$ of $q^2+1$ planes of the spread $\S$ in $\si\cong\PG(5,q)$. In
\cite{barwtgt1}, it is shown that there is an interesting set of cover planes in
$\PG(5,q)$ meeting every element of the tangent splash and contained entirely
within the tangent splash. These cover planes will be useful in our construction later. 

\begin{theorem}\Label{coverplanes}
Let $\ST$ be a tangent splash of $\li$ with centre $T$, and let $[\ST]$ be
the corresponding set of planes in $\si$ in the Bruck-Bose representation in
 $\PG(6,q)$. 
There
are exactly $q^2+q+1$ planes of $\si\cong\PG(5,q)$ that meet every plane of
$[\ST]$, called {\bf cover planes}. These cover planes
each meet the centre $[T]$ in distinct  lines, and meet every other plane of $[\ST]$ in 
distinct points, and hence are contained entirely within the splash. 
\end{theorem}

We will also need the following result about cover planes. 

\begin{theorem}\Label{splash-coversets-T5} Let $[\ST]$ be a tangent splash in the Bruck-Bose representation of
$\PG(2,q^3)$ in $\PG(6,q)$. Let $\pi$ be a  plane which meets the centre $[T]$ in a
  line, and meets
  three further elements $[U],[V],[W]$ of $[\ST]$, where
  $[T],[U],[V],[W]$ are not in a common $2$-regulus. Then 
  $\pi$ is a cover plane of $[\ST]$.
\end{theorem}

\section[Structure]{Structure of the splash in $\PG(6,q)$}\Label{sect:structure}

Let $\Bpi$ be an \orsp\ of $\PG(2,q^3)$ tangent to $\li$ at the point $T$ with tangent splash $\ST$.
By
Theorem~\ref{FFA-subplane-tangent}, in $\PG(6,q)$, $\Bpi$ corresponds to a ruled
surface $\bbb$ with special conic directrix $\C$ in the spread element
$[T]$. Further, $\ST$ corresponds to a set $[\ST]$ of $q^2+1$ planes of the
regular spread $\S$. 
In this section we work in the Bruck-Bose representation in $\PG(6,q)$ and
investigate the interaction between the  ruled surface $\bbb$ and the
tangent splash $[\ST]$, and in particular with the cover planes of 
$[\ST]$. The interaction is complex, so we begin with a verbal description
and diagram before stating the main result.

In $\PG(2,q^3)$, let $\ell$ be an \orsl\ of $\Bpi$ not through $T$, so $\ell$ is
disjoint from $\li$. Let $\overline\ell$ be the extension of $\ell$ to
$\PG(2,q^3)$, and let $\overline\ell\cap\li=L$. 
By Theorem~\ref{sublinesinBB}, in $\PG(6,q)$, $\ell$ corresponds to a special
twisted cubic $[\ell]$ in the 3-space $[\overline\ell]$ about
the spread element $[L]$. 

We will show that the following geometric relationship holds in $\PG(6,q)$.
Through a point $P\in[\ell]$,
there is a unique tangent line to the twisted cubic $[\ell]$; it meets
$\si$ in a point of $[L]$ which
we denote $I_{P,\ell}$.
The set of $q+1$ points $\D=\{I_{P,\ell}\st P\in[\ell]\}$ in $[L]$ is called the {\em
  \tangentshadow} of $[\ell]$. 
We will show that through each point of $\D$, there is a unique cover plane
$\deltam_m$ of the tangent splash $[\ST]$. The cover plane $I_m$ meets $[T]$
in a line, and we show that the resulting set of $q+1$ lines in $[T]$ are
the tangent lines of the conic directrix $\C$ of $[\Bpi]$. 
We will show that the cover plane $\deltam_m$ can be constructed from the points $I_{P,\ell}$
as follows. Let $m$ be a line of $\Bpi$ through
$T$ with points $T,P_1,\ldots,P_q$, see Figure~\ref{definingIPl}. Label the
lines through $P_1$ by $\ell_1,\ldots,\ell_q$. 
We show that the points $\{I_{P_1,\ell_1},\ldots, I_{P_1,\ell_q}\}$ lie on a line
$I_{P_1}$ that meets $[T]$, and that  
the lines $\{I_{P_1},\ldots,I_{P_q}\}$ lie in a plane  $\deltam_m$ as required. 

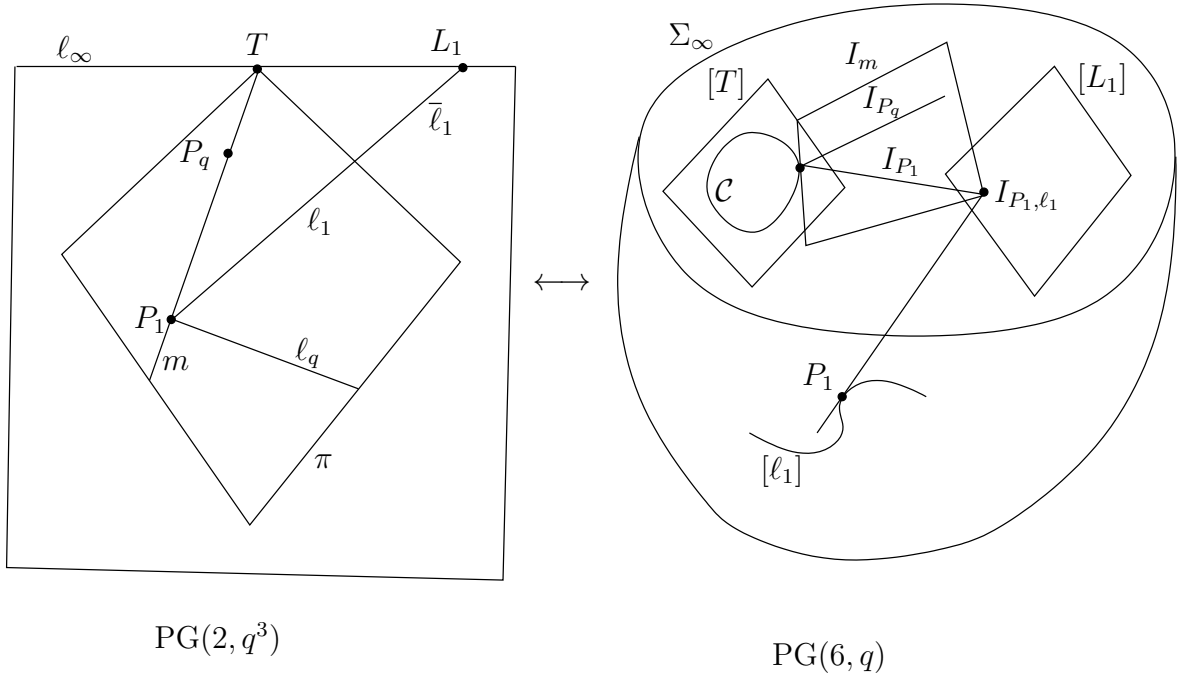
\begin{figure}[ht]
\centering
\input{define-the-I.pdftex_t}
\caption{Defining the points $I_{P,\ell}$, lines $I_P$ and planes $I_m$}\label{definingIPl}
\end{figure}

The main result is stated now.

\begin{theorem}\Label{interaction}
Let $\Bpi$ be an \orsp\ of $\PG(2,q^3)$ tangent to $\li$ at the point
$T$ with tangent splash $\ST$. In $\PG(6,q)$, $\Bpi$ corresponds to a ruled surface $\bbb$ with conic
directrix $\C$ in the spread element $[T]$. 
\begin{enumerate}
\item Let $P$ be a point in $\Bpi\setminus\{T\}$. In $\PG(6,q)$, the $q$ points
  $\{I_{P,\ell}\st \ell$ is  an \orsl\ of $\Bpi$, $P\in\ell, T\notin\ell\}$
    lie on a line denoted $I_P$ that meets
    $[T]$ in a point of $\C$. 
\item Let $m$ be an \orsl\ of $\Bpi$ through $T$. In $\PG(6,q)$, the $q$ lines $\{I_P\st
  P\in m, P\neq T\}$ lie in a
  cover plane $\deltam_m$ of $[\ST]$ that meets $[T]$ in a tangent line of $\C$.
\item Let $L\in\ST\setminus\{T\}$, and let $\ell$ be the unique \orsl\ of
  $\Bpi$ whose extension contains $L$. 
 In $\PG(6,q)$, the $q+1$ planes $\{\deltam_m\st m$ is an \orsl\ of $\Bpi$
  through $T\}$ meet $[L]$ in the \tangentshadow\ 
  $\D=\{I_{P,\ell}\st P\in[\ell]\}$ of $[\ell]$, and  meet $[T]$ in the $q+1$ tangent lines of $\C$.
\end{enumerate}
\end{theorem}

We prove this result using coordinates.
By Theorem~\ref{transitiveontangentsubplanes}, we can
without loss of generality prove this for the \orsp\ $\Bpi$ coordinatised in
Section~\ref{coord-subplane}. 
By Theorem~\ref{FFA-subplane-tangent}, the \orsp\ $\Bpi$ corresponds to a ruled
surface $\bbb$ in $\PG(6,q)$ with a special conic directrix $\C$ in the spread element
$[T]$. We will need the coordinates of this conic, and the next result
calculates them. We use the following notation
\begin{eqnarray*}
\minustheta(e)&=&(e+\tau)^{\r^2+\r}=(e+\tau^q)(e+\tau^{q^2}),\\ 
\plustheta(e)&=&(e+\tau)^{\r^2+\r+1}=(e+\tau)(e+\tau^q)(e+\tau^{q^2}),
\end{eqnarray*}
where $e\in\GF(q)$.
Note that since $\plustheta(e)^\r=\plustheta(e)$, it follows that
$\plustheta(e)\in\GF(\r)$. The next lemma  uses the generalised Bruck-Bose map $\sigma$ defined in Section~\ref{BBintro}.

\begin{lemma}\Label{cpoints} The conic directrix $\C$ in $[T]$ in $\PG(6,q)$
  of the \orsp\ $\Bpi$ coordinatised in Section~$\ref{coord-subplane}$ has points 
$C_e=\sigma(\minustheta(e)\tau,0,0)=(t_0,e^2+et_2,-e,\ 0,0,0,\allowbreak\
0)$ for $e\in\GF(\r)$,  and $C_\infty=\sigma(\tau,0,0)=(0,1,0,\ 0,0,0,\ 0)$.
Further, the tangent line to $\C$ in $[T]$ at the point $C_e$ is given by the line
$C_eT_e$ where $T_e=\sigma(\minustheta(e)^2\tau,0,0)$ for $e\in\GF(q)$, and $T_\infty=\sigma(\tau^2,0,0)$.
\end{lemma}

\begin{proof} 
We use the notation for points and lines of $\Bpi$ given in Section~\ref{coord-subplane}.
By Theorem~\ref{sublinesinBB}, the \orsl\
$m_e$ of $\Bpi$ corresponds  in $\PG(6,q)$ to an affine line $[m_e]$ contained in the
ruled surface $[\Bpi]$, $e\in\GF(q)\cup\{\infty\}$. Hence the $q+1$
generators of the ruled surface $\bbb$ are the lines
$[m_e]$, $e\in\GF(q)\cup\{\infty\}$. Thus the points 
of the conic directrix $\C$ in $[T]$ are $C_e=[m_e]\cap\si$,  $e\in\GF(\r)\cup\{\infty\}$. 

If $e=\infty$, consider the two points $U_0=(1,1,1)$ and
$U_1=(1+\tau,1,1)$ of $\Bpi$
which lie on the line $m_\infty$. So
$[m_\infty]=\langle\sigma(1,1,1),\sigma(1+\tau,1,1)\rangle$, and subtracting gives
$C_\infty=[m_\infty]\cap\si=\sigma(\tau,0,0)=(0,1,0,\ 0,0,0,\ 0)$. 
For $e\in\GF(\r)$, consider the two distinct points
  $P_{0,e}=(e,e,e+\tau)$, $P_{1,e}=(e+\tau,e,e+\tau)$ of $\Bpi$ on the line $m_e$.
  Multiplying by $\minustheta(e)$ we have
  $P_{0,e}=(e\minustheta(e),e\minustheta(e),\plustheta(e))$ and
  $P_{1,e}=(\plustheta(e),e\minustheta(e),\plustheta(e))$.  Mapping these to points
  in $\PG(6,q)$, and subtracting gives $C_e=\sigma(\minustheta(e)\tau,0,0)$ as
  required.   
Note that
  $\minustheta(e)\tau=\tau(e+\tau^\r)(e+\tau^{\r^2})=e^2\tau+e\tau(\tau^\r+\tau^{\r^2}))+\tau\tau^\r\tau^{\r^2}=e^2\tau+e\tau(t_2-\tau)+t_0=t_0+(e^2+et_2)\tau-e\tau^2$, and so 
$\sigma(\minustheta(e)\tau,0,0)=(t_0,e^2+et_2,-e,\ 0,0,0,\allowbreak\
0)$. 

To calculate the tangents for $\C$ in $[T]$, let
$T_e=\sigma(\minustheta(e)^2\tau,0,0)$ for $e\in\GF(q)$, and consider the line
$t_e=C_eT_e=\{ rC_e+sT_e\st r,s\in \GF(\r)\cup\{\infty\}\}$ in $[T]$. We
show that the point $C_f$ of $\C$
is on the line $t_e$ if and only if $e=f$. Suppose $C_f=rC_e+sT_e$ for some $r,s\in
\GF(\r)\cup\{\infty\}$, then 
 $r\minustheta(e)\tau+s\minustheta(e)^2\tau=\minustheta(f)\tau$. Multiplying by $(e+\tau)^2(f+\tau)/\tau$ yields
  $r(f+\tau)(e+\tau)\plustheta(e)+s(f+\tau)\plustheta(e)^2=\plustheta(f)(e+\tau)^2$, and so $(ref\plustheta(e)+sf\plustheta(e)^2)+(re\plustheta(e)+rf\plustheta(e)+s\plustheta(e)^2)\tau+r\plustheta(e)\tau^2=\plustheta(f)e^2+2e\plustheta(f)\tau+\plustheta(f)\tau^2$. Equating the coefficients (in $\GF(q)$) of $\tau^2,\tau,1$ gives  
\begin{eqnarray}
r\plustheta(e)&=&\plustheta(f)\label{eqnef1}\\
re\plustheta(e)+rf\plustheta(e)+s\plustheta(e)^2&=&2e\plustheta(f)\label{eqnef2}\\
ref\plustheta(e)+sf\plustheta(e)^2&=&\plustheta(f)e^2.\label{eqnef3}
\end{eqnarray}
Note that as $\plustheta(f)\neq0$, we have $r\neq0$. Substituting
(\ref{eqnef1}) into (\ref{eqnef2}) gives
$s\plustheta(e)=r(e-f)$. Substituting this and (\ref{eqnef1}) into (\ref{eqnef3})
gives $r\plustheta(e)(e-f)^2=0$, and so $e=f$.
  Thus  the line $C_eT_e$ meets $\C$ in the point $C_e$, and so is a tangent
  to $\C$ for $e\in\GF(q)$.

Now consider the case of $C_\infty=\sigma(\tau,0,0)$,
  $T_\infty=\sigma(\tau^2,0,0)$. In a similar manner to the above, if
  $r\tau+s\tau^2=\minustheta(e)\tau$ we multiply by $(e+\tau)/\tau$ to get
  $r(e+\tau)+s(e+\tau)\tau=\plustheta(e)$. Equating coefficients of $\tau^2$
  and then $\tau$ gives $s=0$
  and $r=0$. Hence $C_\infty T_\infty$ is a tangent as required.
Note that if we put $e=\infty$ in the expression for
$T_e=\sigma(\minustheta(e)^2\tau,0,0)$, we get $\sigma(\tau,0,0)$, which is
$C_\infty$. So we need
to define $T_\infty$ differently to ensure that $C_\infty
T_\infty$ is a line.

Finally, we use these coordinates to verify that $\C$ is a special conic. 
If we regard $[T]$ as $\PG(2,q)$ with points $(x,y,z)$, then $\C$ has
equation $t_0z^2-xy-t_2xz=0$. This is easy to verify as the 
point $(t_0,e^2+et_2,x-e)$ corresponding to $C_e$ satisfies the equation,
and the point $(0,1,0)$ corresponding to $C_\infty$ also satisfies the equation.
To prove that $\C$ meets the transversals of $\S$, 
we recall from Lemma~\ref{transversaleqn} that one transversal point is 
$R=(t_1+t_2\tau-\tau^2,t_2-\tau,-1)$.  Substituting $R$ into the equation of
$\C$ gives
$t_0(-1)^2-(t_2-\tau)(t_1+t_2\tau-\tau^2)-t_2(-1)(t_1+t_2\tau-\tau^2)=t_0+t_1\tau+t_2\tau^2-\tau^3=0$,
hence
$R\in\C$. Since the equation of
$\C$ is over $\GF(\r)$, the points $R^\r$, $R^{\r^2}$ are also on $\C$.
\end{proof}

Theorem~\ref{interaction} is proved in the next lemma using the following three
steps. Part 1 calculates the coordinates of the points
 $I_{P,\ell}$ for all pairs $(P,\ell)$, where $P$ is a point of
 $\Bpi$ incident with an \orsl\ $\ell$ of $\Bpi$ ($P\neq T$, $\ell$ not through $T$).
Part 2 shows that by varying the \orsl\ $\ell$ through $P$ we construct a set of $q$ collinear points
$\{I_{P,\ell}\st \ell {\rm \ a\ line\ of\ }\Bpi{\rm \ through\ } P, T\notin\ell\}$. These
points lie on a line
denoted $I_{P}$ that meets $[T]$ in a point of the conic directrix $\C$. In part 3
of the  lemma we show that the set of $q$ lines 
$\{I_P\st P {\rm \ is\ on\ a\ line\ }m{\rm \ of\ }\Bpi{\rm \  through\ } T,
P\neq T\}$
lie on a cover plane $\deltam_m$ that meets $[T]$ in a
tangent line of $\C$.

\begin{lemma}\Label{coords-of-the-IP}  Let $\Bpi$ be the \orsp\
  coordinatised in Section~$\ref{coord-subplane}$. Consider the \orsls\
$\ell_{e,d,f}$ of $\Bpi$ with points $P_{e,d}$ and $U_f$, $e,d,f\in\GF(q)$.
 Then
\begin{enumerate}
\item $I_{P_{e,d},\ell_{e,d,f}}=\sigma(((1-d+ef)\tau+f\tau^2)\minustheta(e)^2,\ \tau\minustheta(e)^2,\ 0)$, $(e,d,f\in\GF(\r))$.\\
$I_{U_f,\ell_{e,d,f}}=\sigma((1-d+ef)\tau+f\tau^2,\ \tau, \ 0)$, $(e,d,f\in\GF(\r))$.
\item
For fixed $e,d\in\GF(q)$, the set
$I_{P_{e,d}}=\{I_{P_{e,d},\ell_{e,d,f}}\st f\in \GF(q)\}\cup\{C_e\}$ is a line,
For fixed $f\in\GF(q)$, the set $I_{U_f}=\{I_{U_f,\ell_{e,d,f}}\st e,d\in\GF(\r)\}\cup\{C_\infty\}$ is a
line.
\item
For fixed $e\in\GF(q)$, the set 
 $\deltam_{m_e}=\{I_{P_{e,d}}\st d\in \GF(q)\}\cup\{C_eT_e\}$ is a cover plane of $[\ST]$,
\\
The set $\deltam_{m_\infty}=\{I_{U_f}\st f\in\GF(\r)\}\cup\{C_\infty T_\infty\}$
is a cover plane of $[\ST]$.
\end{enumerate}
\end{lemma}

\begin{proof} We use the coordinates for points and lines of $\Bpi$ given in
  Table~\ref{table-coords}.  
Note that every line of $\Bpi$ not through
$T$ is $\ell_{e,d,f}$ for some $e,d,f\in\GF(q)$, and every point $P$ of
$\Bpi$ not on the line $m_\infty$
can be written as $P_{e,d}$ for some $e,d\in\GF(q)$. Further, the point
$P_{e,d}$ lies on the line $\ell_{e,d,f}$ for $f\in\GF(q)$. Hence the pairs
$(P_{e,d},\ell_{e,d,f})$ and $(U_f,\ell_{e,d,f})$ cover all the pairs
$(P,\ell)$ such that $P$ is a point of $\Bpi$ distinct from $T$, and $\ell$
is an \orsl\ of $\Bpi$ not through $T$. Hence it suffices to calculate
$I_{P,\ell}$ for these pairs. 

Consider a line $\ell_{e,d,f}$ for fixed $e,d,f\in\GF(q)$. The points
of $\ell_{e,d,f}$ are 
$R_{e,d,f,h}$ for $h\in\GF(q)\cup\{\infty\}$. Note that
$R_{e,d,f,e}=P_{e,d}$ and $R_{e,d,f,\infty}=U_f$.
To calculate the
coordinates of the point 
$I_{P_{e,d},\ell_{e,d,f}}$ in $\PG(6,q)$, we need to look at the tangent
line to the twisted cubic $[\ell_{e,d,f}]$
at the point $P_{e,d}$. We consider the secant line $P_{e,d}R_{e,d,f,h}$ 
of $[\ell_{e,d,f}]$ and
calculate where it meets $\si$. Then letting $h=e$ will give us the
coordinates of 
$I_{P_{e,d},\ell_{e,d,f}}$. To find where this secant line meets $\si$, we
take the coordinates of $P_{e,d}$ and $R_{e,d,f,h}$ in $\PG(2,q^3)$, and
write them with last coordinate in $\GF(q)$. This allows us to use the generalised
Bruck-Bose map $\sigma$ defined in Section~\ref{BBintro} to convert them to
coordinates in $\PG(6,q)$.

Let $X=\minustheta(e)\plustheta(h)P_{e,d}$ and $Y=\plustheta(e)\minustheta(h)R_{e,d,f,h}$ so
\begin{eqnarray*}
X&=&\sigma((e+d\tau)(h+\tau)\minustheta(e)\minustheta(h),\ e(h+\tau)\minustheta(e)\minustheta(h),\ \plustheta(e)\plustheta(h))\\
Y&=&\sigma((h+(fh-fe+d)\tau)(e+\tau)\minustheta(e)\minustheta(h),\ h(e+\tau)\minustheta(e)\minustheta(h),\ \plustheta(e)\plustheta(h))
\end{eqnarray*}
and
\begin{eqnarray*}
X-Y
&=&(e-h)\sigma(((1-d+ef)\tau+f\tau^2)\minustheta(e)\minustheta(h),\ \tau\minustheta(e)\minustheta(h), \ 0)\\
&\equiv&\sigma(((1-d+ef)\tau+f\tau^2)\minustheta(e)\minustheta(h),\ \tau\minustheta(e)\minustheta(h), \ 0)
\end{eqnarray*}
as $e-h\in\GF(\r)$.  Now 
letting $h=e$, 
we have 
\[
I_{P_{e,d},\ell_{e,d,f}}=\sigma(((1-d+ef)\tau+f\tau^2)\minustheta(e)^2,\ \tau\minustheta(e)^2,\ 0)
\]
as required. 

Now consider the point $U_f=R_{e,d,f,\infty}$ on 
$\ell_{e,d,f}$.  
Similar to the above let $X=\plustheta(h)U_f$ and $Y=\minustheta(h)R_{e,d,f,h}$.  Then
\begin{eqnarray*}
X&=&\sigma((1+f\tau)(h+\tau)\minustheta(h),\ (h+\tau)\minustheta(h),\ \plustheta(h))\\
Y&=&\sigma((h+(fh-fe+d)\tau)\minustheta(h),\ h\minustheta(h),\ \plustheta(h))
\end{eqnarray*}
and 
$ X-Y=\sigma(((1-d+ef)\tau+f\tau^2)\minustheta(h),\ \tau\minustheta(h),\
0)$, so
$$
\frac 1{h^2}(X-Y)=\sigma\Biggl(((1-d+ef)\tau+f\tau^2)\!\left(1+\frac {\tau^\r}h\right)\!\!\left(1+\frac {\tau^{\r^2}}h\right),
\tau\!\left(1+\frac {\tau^\r}h\right)\!\!\left(1+\frac {\tau^{\r^2}}h\right)\!, 0\Biggr).$$
Taking the limit of this point as $h\rightarrow\infty$ gives
$I_{U_f,\ell_{e,d,f}}=\sigma((1-d+ef)\tau+f\tau^2,\ \tau, \ 0)$.
This completes the proof of part 1.

For part 2, we fix the point $P_{e,d}$ (so fix $e,d\in\GF(q)$) and look at
the lines of $\Bpi$ through $P_{e,d}$, but not through $T$, namely the $q$ lines
$\ell_{e,d,f}$ for $f\in\GF(q)$. We show that the $q$ 
points $I_{P_{e,d},\ell_{e,d,f}}$ for $f\in\GF(q)$ lie on a line that contains
the point $C_e=\sigma(\tau\minustheta(e),0,0)$ of the conic directrix $\C$.  First note that $I_{P_{e,d},\ell_{e,d,0}}=\sigma((1-d)\tau\minustheta(e)^2,\tau\minustheta(e)^2,0)$.  Now
\begin{eqnarray*}
I_{P_{e,d},\ell_{e,d,f}}&=&\sigma(((1-d+ef)\tau+f\tau^2)\minustheta(e)^2,\ \tau\minustheta(e)^2,\ 0)\\
&=&\sigma(f(e+\tau)\tau\minustheta(e)^2,\ 0,\ 0)+\sigma((1-d)\tau\minustheta(e)^2,\tau\minustheta(e)^2,0)\\
&=&\sigma(f\tau\plustheta(e)\minustheta(e),\ 0,\ 0)+I_{P_{e,d},\ell_{e,d,0}}\\
&=&f\plustheta(e)C_e+I_{P_{e,d},\ell_{e,d,0}},
\end{eqnarray*}
as $f\plustheta(e)\in\GF(q)$. Hence for fixed $e,d\in\GF(q)$, the points
$I_{P_{e,d},\ell_{e,d,f}}$ for $f\in\GF(q)$ all lie on the line joining
$C_e$ and $I_{P_{e,d},\ell_{e,d,0}}$, as required.  

Now we consider the remaining points $U_f$ of $\Bpi$.  For fixed
$f\in\GF(q)$, the lines $\ell_{e,d,f}$, $e,d\in\GF(q)$ consist of  the $q$ lines of
$\Bpi$ through $U_f$, not through $T$.
We show that the $q$ points 
$I_{U_f,\ell_{e,d,f}}$, $e,d\in\GF(\r)$, lie on a line through $C_\infty=\sigma(\tau,0,0)$. 
First note that $I_{U_f,\ell_{0,0,f}}=\sigma(\tau+f\tau^2,\tau,0)$.  Now
\begin{eqnarray*}
I_{U_f,\ell_{e,d,f}}&=&\sigma((1+ef-d)\tau+f\tau^2,\ \tau, \ 0)\\
&=&\sigma((ef-d)\tau+(\tau+f\tau^2),\ \tau, \ 0)\\
&=&(ef-d)\sigma(\tau,0,0)+\sigma(\tau+f\tau^2,\tau,0)\\
&=&(ef-d)C_\infty+I_{U_f,\ell_{0,0,f}},
\end{eqnarray*}
as $ef-d\in\GF(q)$.
Hence for fixed $f\in\GF(q)$, the points $I_{U_f,\ell_{e,d,f}}$ for $e,d\in\GF(q)$ all lie on the
line joining the points $C_\infty$ and $I_{U_f,\ell_{0,0,f}}$,
as required.

For part 3, note that for fixed $e\in\GF(q)$, the lines $I_{P_{e,d}}$ for $d\in \GF(\r)$ and the
tangent line $C_eT_e$ all share the point $C_e$.  Also note that the
$I_{P_{e,d}}$ for $d\in \GF(\r)$ are distinct lines, since they meet
different splash elements.  Thus to show that they lie in a plane, it is
sufficient to find a line not through $C_e$ that meets $C_eT_e$ and meets each $I_{P_{e,d}}$
for $d\in \GF(\r)$. 
The line
$n=\langle\sigma(\tau\minustheta(e)^2,\tau\minustheta(e)^2,0),T_e\rangle$
does this as the line $I_{P_{e,d}}$ contains the point
$I_{P_{e,d},\ell_{e,d,0}}$ (for any $d\in\GF(q)$) and
$$
I_{P_{e,d},\ell_{e,d,0}}
=\sigma((1-d)\tau\minustheta(e)^2,\tau\minustheta(e)^2,0)
=\sigma(\tau\minustheta(e)^2,\tau\minustheta(e)^2,0)-dT_e\in
n.
$$
Hence $I_{m_e}$ is a plane. It is a cover plane of $[\ST]$ as it is
completely contained in $[\ST]$.

Finally, we consider the set $\deltam_{m_\infty}=\{I_{U_f}\st
f\in\GF(\r)\}\cup\{C_\infty T_\infty\}$. In a similar manner to the
previous case, we have
$$
I_{U_f,\ell_{0,0,f}}
=\sigma(\tau+f\tau^2,\tau,0)
=fT_\infty+\sigma(\tau,\tau,0).
$$
So $\langle T_\infty,\sigma(\tau,\tau,0)\rangle$ is a line not through $C_\infty$ that meets
$I_{U_f}$, $f\in\GF(q)$ and $C_\infty T_\infty$, hence $I_{m_\infty}$ is a
plane. 
\end{proof}

Note that this lemma completes the proof of Theorem~\ref{interaction}.
We will later need the coordinates of the \tangentshadow\  for the line $\ell_1=\ell_{e,1,0}$
calculated in this
lemma, so we state this as a corollary. 

\begin{corollary}\Label{limage}
The \orsl\ 
 $\ell_1=\{P_{e,1}\st e\in\GF(q)\}\cup\{P_\infty=U_0\}$ has
\tangentshadow\ 
$\D=\{D_e=I_{P_{e,1},\ell_{e,1,0}}=\sigma(0,\minustheta(e)^2\tau,0)\st e\in\GF(\r)\}\cup\{D_\infty=\sigma(0,\tau,0)\}$.
\end{corollary}

\section{The tangent subspace of a point}\Label{sect:tgtsubspace}

Let $\Bpi$ be an \orsp\ in $\PG(2,\r^3)$ tangent to $\li$ at the point $T$.  For
each affine point $P$ of $\Bpi$, 
we can construct an \orsp\ $P^\perp$ that contains $P$ and is secant to $\li$ as follows. Let
$\ell_1,\ldots,\ell_{q+1}$ be the $q+1$ lines of $\Bpi$ through
$P$. Then $m=\{\ell_i\cap\li\st i=1,\ldots,q+1\}$ is an \orsl\ of $\li$
through $T$. Now
$m$ and $PT\cap\Bpi$ are two \orsls\  through $T$, and so lie in a unique
\orsp\ which we denote by $P^\perp$, and call the tangent subspace.
Recall that in $\PG(6,q)$, $\Bpi$ corresponds to a ruled surface $\bbb$ and $P^\perp$
corresponds to a plane $[P^\perp]$.
In \cite{barwtgt1} it was shown that in $\PG(6,q)$, the plane 
$[P^\perp]$ is the tangent space to the ruled surface $\bbb$ at the point $P$. We can use
Theorem~\ref{interaction} to 
investigate the structure of $[P^\perp]$ in $\PG(6,q)$ in more detail.

Label the \orsls\ of $\Bpi$ through $P$ by $TP,\ell_1,\ldots,\ell_q$. In
$\PG(6,q)$, the \orsl\ $\ell_i$
corresponds to a twisted cubic $[\ell_i]$, $i=1,\ldots,q$.
We show that
the $q+1$ lines through
$P$ in the plane $[P^\perp]$ consist of  the
generator line of $\bbb$ through $P$ and the 
tangent line to the twisted cubic $[\ell_i]$ at the point $P$, for $i=1,\ldots,q$.

\begin{lemma}\Label{Pperpgen}
The plane $[P^\perp]$ contains a generator line of $\bbb$.
\end{lemma}

\begin{proof} Using Theorem~\ref{transitiveontangentsubplanes} and \cite[Lemma~\ref{TS-grouplemma}]{barwtgt1}, we can 
without loss of generality prove this for the \orsp\ $\Bpi$ coordinatised
in Section~\ref{coord-subplane} and the point $P_{e,d}$ of $\Bpi$. 
In  \cite[Corollary 10.4]{barwtgt1}, it is shown that the point
$C_e=\sigma(\tau\minustheta(e),0,0)$ lies in the plane $[P_{e,d}^\perp]$. 
By Lemma~\ref{cpoints}, $C_e$ is a
point of the conic directrix of $\bbb$, so
the plane $[P_{e,d}^\perp]$ meets the conic directrix. The line
$P_{e,d}C_e$ is in the plane $[P_{e,d}^\perp]$, so in $\PG(2,q^3)$ it
corresponds to the \orsl\ $m_e$ of $\Bpi$. Thus in $\PG(6,q)$, $[P_{e,d}^\perp]$ contains the generator
$P_{e,d}C_e$ of the ruled surface $\bbb$. 
\end{proof}

\begin{theorem}\Label{perpIm}
For each affine point $P$ of $\Bpi$, we have $[P^\perp]=\langle P,I_P\rangle$.
\end{theorem}

\begin{proof} Let $\Bpi$ be the \orsp\ coordinatised in
  Section~\ref{coord-subplane} and let $P_{e,d}$ be a point of $\Bpi$. 
By \cite[Corollary 10.4]{barwtgt1}, $[P_{e,d}^\perp]$ contains
the conic directrix point $C_e$ and
 the point
 $J_{e,d}=\sigma((1-d)\tau\minustheta(e)^2,\allowbreak\tau\minustheta(e)^2,0)$.  
Using Lemma~\ref{coords-of-the-IP}(1), we have
$J_{e,d}=I_{P_{e,d},\ell_{e,d,0}}$. By Lemma~\ref{coords-of-the-IP}(2),
$I_{P_{e,d}}$ contains $C_e$ and $J_{e,d}$. As $P_{e,d}\notin I_{P_{e,d}}$, it follows that  $[P_{e,d}^\perp]=\langle
P_{e,d},C_eJ_{e,d}\rangle=\langle P_{e,d},I_{P_{e,d}}\rangle$.
\end{proof}

\begin{corollary}
  In $\PG(2,q^3)$, let $\Bpi$ be an \orsp\ tangent to $\li$ at $T$, and let
  $P$ be an affine point of $\Bpi$. The lines of $\Bpi$ through $P$
  correspond in $\PG(6,q)$ to $q$ twisted cubics $[\ell_i]$,
  $i=1,\ldots,q$, and one generator line of $\bbb$ (the ruled surface
  corresponding to $\Bpi$). The tangent to the twisted cubic $[\ell_i]$ at $P$ lies in the
  tangent subspace $[P^\perp]$ for $i=1,\ldots,q$.
\end{corollary}

\section{Constructing an \orsp\ from a splash and a subline}\Label{sect:construct}

In $\PG(2,q^3)$, let $\ST$ be a tangent splash of $\li$ with centre $T$ and let $\ell$ be an \orsl\ disjoint from
$\li$ whose extension $\overline\ell$ to $\PG(2,q^3)$ meets $\li$ in a
point $L=\overline\ell\cap\li$ of $\ST\backslash\{T\}$. 
By Theorem~\ref{splashextsubplane}, there is a unique \orsp\
tangent to $\li$  that contains
$\ell$ and has tangent splash $\ST$. 
In this section we give a geometric construction of this subplane.

We 
work in the
Bruck-Bose representation of $\PG(2,q^3)$ in $\PG(6,q)$, and also use the
cubic extension $\PG(6,q^3)$ of $\PG(6,q)$. 
We use the following notation for the cubic extension $\PG(6,q^3)$. If
$\K$ is a subspace or a curve of $\PG(6,q)$ then denote by $\K^*$ the natural extension
of $\K$ to a subspace or curve of $\PG(6,q^3)$. 

In $\PG(6,q)$, the \orsl\ $\ell$ 
 corresponds to a special twisted cubic
$[\ell]$ in a
3-space $[\overline\ell]$ about a spread element $[L]$ (by Theorem~\ref{sublinesinBB}). 
In order to construct the required \orsp\ of $\PG(2,q^3)$, we 
use Theorem~\ref{subplane-tangent-iff} and
construct a ruled surface $\mathcal V$ in $\PG(6,q)$ with 
twisted cubic directrix $[\ell]$ and a conic directrix $\C$ in the centre $[T]$, such
that in the cubic
extension $\PG(6,q^3)$, $\mathcal V^*$ contains the transversals of the
regular spread $\S$. 

Recall from Theorem~\ref{coverplanes} that the tangent splash $[\ST]$ in
$\PG(6,q)$ has an associated set of $q^2+q+1$ cover planes that each meet the
centre $[T]$ in distinct lines, and meet every other plane of $[\ST]$ in distinct points. 
Our construction will exploit the rich geometrical nature of the cover planes
of $[\ST]$ and their interaction with the shadow points of the twisted
cubic $[\ell]$ (this interaction is described in Section~\ref{sect:structure}).
The construction is now stated, then is proved using a series of lemmas. 

\begin{construction}\Label{construction}
In $\PG(2,q^3)$, let $\ST$ be a tangent splash of $\li$ with centre $T$ and let $\ell$ be an \orsl\ disjoint from
$\li$ whose extension to $\PG(2,q^3)$ meets $\li$ in a point of
$\ST\backslash\{T\}$. Then in $\PG(6,q)$:
\begin{enumerate}
\item For each point $N_i$, $i=1,\ldots,q+1$, of the twisted cubic $[\ell]$, let $D_i$
  be the intersection of the tangent to $[\ell]$ at $N_i$ with $[L]$. The
  set of points $\D=\{D_1,\ldots,D_{q+1}\}$ in $[L]$ is called the 
    \tangentshadow\  of $[\ell]$.
\item Map the \tangentshadow\  $\D$ in $[L]$ to a special conic $\C$ in $[T]$ by
\begin{enumerate}
\item Through each point $D_i\in\D$ there is a unique cover plane of $\ST$,
  this cover plane meets $[T]$ in a line $t_i$.
\item The lines $t_1,\ldots,t_{q+1}$ are tangent lines to a unique special conic $\C$ of $[T]$.
\end{enumerate}
\item Let $t_i\cap\C=C_i$, then the lines $N_iC_i$, $i=1,\ldots,q+1$, form a ruled surface
  that corresponds in $\PG(2,q^3)$ to the unique \orsp\ tangent to $\li$
  containing $\ell$ and with tangent splash $\ST$.
\end{enumerate}
\end{construction}


We begin with three lemmas that investigate the shadow of an \orsl.
We first determine the
structure of a \tangentshadow\ in a spread element $[L]$, and show that it is
different for $q$ even and $q$ odd. 

\begin{lemma}\Label{limagelookslike}
Let $\ell$ be an \orsl\ of $\PG(2,q^3)$ disjoint from $\li$, and let $\D$
be the shadow of $\ell$ in $\PG(6,q)$.
\begin{enumerate}
\item For $q$ even, the \tangentshadow\ $\D$ is a special conic of $[L]$.
\item For $q$ odd, the \tangentshadow\ $\D$ has the property that no 3 points of $\D$ lie in a special conic of $[L]$.
\end{enumerate}
\end{lemma}

\begin{proof}
Let $\ell$ be an \orsl\ of $\PG(2,q^3)$ disjoint from $\li$, such that its extension
$\overline\ell$ meets $\li$ in a point $L$ of $\ST\setminus\{ T\}$. By
Theorem~\ref{sublinesinBB}, in $\PG(6,q)$,
$\N=[\ell]$ is a special twisted cubic in a 3-space
$[\overline\ell]$ that meets $\si$ in the spread element $[L]$. 
Let the point $N_i\in\N$, $i=1,\ldots,q+1$, have tangent $n_i$ to $\N$ and
let $D_i=n_i\cap[L]$. So the \tangentshadow\ of $[\ell]$ is $\D=\{D_1,\ldots,D_{q+1}\}\subset[L]$.
In the
cubic extension $\PG(6,q^3)$, let $[L]^*$ meet the transversal lines $g,g^q,g^{q^2}$ of
the regular spread $\S$ in
the points $P,P^q,P^{q^2}$ respectively. As $\N$ is special, $P,P^q,P^{q^2}$
lie in $\N^*$, the extension of $\N$ to $\PG(6,q^3)$. 

If $q$ is even, then the tangents to a twisted cubic lie in a regulus
\cite[Theorem 21.1.2]{hirs85}. This regulus lies in a unique hyperbolic quadric
$\H$ of the 3-space $[\overline\ell]$. As the twisted cubic $\N$ is disjoint
from $[L]$, $\H$ meets $[L]$ in a non-degenerate conic, so the
\tangentshadow\ 
$\D$ is a non-degenerate conic. 
In the cubic
extension $\PG(6,q^3)$, $\H^*$ contains the twisted cubic $\N^*$,
and so $\H^*$ contains the transversal points $P,P^q,P^{q^2}$. So the conic
$\D^*$ contains $P,P^q,P^{q^2}$, and hence $\D$ is a special conic of $[L]$.

Now suppose $q$ is odd.
Consider the points $N_1,N_2\in\N$ with tangents $n_1,n_2$. By
\cite[Lemma 21.1.6(Cor 2)]{hirs85}, $\N,n_1,n_2$ lie in a hyperbolic
quadric $\H$. 
By \cite[Lemma 2.6]{barw13}, the points $D_1,D_2$ lie in a unique special
conic $\C$ of $[L]$. As $\H^*$ and $\C^*$ both contain the five points
$D_1,D_2,P,P^q,P^{q^2}$, $\H^*$ contains $\C^*$ and so $\H$ contains $\C$.

We now show that $\C$ contains no further point of the \tangentshadow\ 
$\D$.  Suppose to the contrary that $\C$ contains the point $D_3$ (so
$D_3$ lies on the tangent
$n_3$).  
By \cite[Lemma 21.1.7]{hirs85}, the tangents $n_1,n_2$ lie in one regulus
$\R$ of $\H$, and $\R$ consists of chords of $\N$. 
As $D_3\in
\C\subset \H$, it follow that $D_3$ is on line $n$  of the regulus $\R$.
Note that $n$ is a chord of $\N$ as $\R$ consists of chords.
By \cite[Lemma 21.1.8]{hirs85}, $\R$ contains at most two tangents of $\N$,
hence as $\R$ contains $n_1,n_2$, $\R$ does not contain $n_3$. So $n$ is
not the tangent $n_3$. Hence through the point $D_3$, there are two chords of $\N$,
namely $n$ and the tangent $n_3$. This contradicts \cite[Theorem 21.1.9]{hirs85}. 
Hence a special
conic of $[L]$ contains at most two points of a \tangentshadow.
\end{proof}

Next we look at
 the action in $\PG(6,q)$ of a certain Singer cycle acting on the
 \tangentshadow\  of an \orsl, and also acting
on the conic
directrix of an \orsp.

\begin{lemma}\Label{singerorbits}  Let $\Bpi$ be an \orsp\ of $\PG(2,q^3)$
  tangent to $\li$
  and let $\ell$ be an \orsl\ of $\Bpi$ disjoint from $\li$. In the
  Bruck-Bose representation in $\PG(6,q)$,  consider the collineation $\CapitalTheta$ of
  $\PG(6,\r)$ defined in Theorem~{\rm \ref{PGL7r}}.
\begin{enumerate}
\item\Label{soc}
The conic directrix  $\C$ of
$\bbb$ has an orbit of size  $\r^2+\r+1$  under $\langle\CapitalTheta\rangle$.
\item\Label{sod}
The \tangentshadow\ $\D$ of $[\ell]$ has an orbit of size $\r^2+\r+1$ under $\langle\CapitalTheta\rangle$.
\end{enumerate}
\end{lemma}
\begin{proof}
By Theorem~\ref{transitiveontangentsubplanes}, we can without loss of
generality prove this for the \orsp\ coordinatised in
Section~\ref{coord-subplane}. 
The coordinates of the conic directrix $\C$ of $\bbb$ are calculated in Lemma~\ref{cpoints}.
So $\C=\{\sigma(\tau,0,0)\}\cup\{ \sigma(\minustheta(e)\tau,0,0) \st
e\in\GF(\r)\}$, and
$\CapitalTheta^i(\C)=\{\sigma(\tau^{i+1},0,0)\}\cup\{
\sigma(\minustheta(e)\tau^{i+1},0,0) \st e\in\GF(\r)\}$.  Suppose
$\C=\CapitalTheta^i(\C)$ for some $i$ with $1\le i <\r^2+\r+1$. 
Suppose firstly that $\r\ge 3$, so
$\C$ contains at least four points.  Hence there exists
$a,b,e,f\in\GF(\r)$ with $a\ne b$ and $e\ne f$, such that
$\sigma(\minustheta(a)\tau,0,0)\equiv\sigma(\minustheta(e)\tau^{i+1},0,0)$ and 
$\sigma(\minustheta(b)\tau,0,0)\equiv\sigma(\minustheta(f)\tau^{i+1},0,0)$,
thus there exists nonzero $c,d\in\GF(q)$ such that
\begin{eqnarray}
\minustheta(a)\tau=c\minustheta(e)\tau^{i+1},\quad \minustheta(b)\tau=d\minustheta(f)\tau^{i+1}.\label{abef}\end{eqnarray}
Hence $\minustheta(a)/\minustheta(b)=(c/d)\minustheta(e)/\minustheta(f)$ and so $d\minustheta(a)
\minustheta(f)=c\minustheta(b)\minustheta(e)$.  Multiplying both sides by $(a+\tau)(b+\tau)(e+\tau)(f+\tau)$ yields
\begin{eqnarray}\label{anotherone}
d\plustheta(a)\plustheta(f)(be+(b+e)\tau+\tau^2)&=&c\plustheta(b)\plustheta(e)(af+(a+f)\tau+\tau^2).
\end{eqnarray}
Equating the coefficient of $\tau^2$ gives
$d\plustheta(a)\plustheta(f)=c\plustheta(b)\plustheta(e)$, and substituting back into (\ref{anotherone}) gives $be+(b+e)\tau=af+(a+f)\tau$. Equating the coefficient of
$\tau$ gives $b=a+f-e$ and using the constant term gives $(a+f-e)e=af$,
hence $(f-e)(e-a)=0$. As $e\ne f$ we have $e=a$.  So by (\ref{abef}) we
have $\minustheta(a)\tau=c\minustheta(a)\tau^{i+1}$ and so $1/c=\tau^i$. Now
$\tau^i\in\GF(q)$ implies $\r^2+\r+1\bigm| i$, contradicting $1\le i<\r^2+\r+1$.  Hence the orbit of $\C$ under $\langle\CapitalTheta\rangle$ is of size at least $\r^2+\r+1$.  As $\tau^{\r^2+\r+1}\in\GF(\r)$, $\CapitalTheta^{\r^2+\r+1}(\C)=\C$.  Hence the size of the orbit of $\C$  is exactly $\r^2+\r+1$.

Now consider the case $\r=2$.  Let
$\C=\{C_{1,0}=\sigma(\tau,0,0),C_{2,0}=\sigma(\minustheta(0)\tau,0,0),C_{3,0}=\sigma(\minustheta(1)\tau,0,0)\}$
and so
$\CapitalTheta^i(\C)=\{C_{1,i}=\sigma(\tau^{i+1},0,0),C_{2,i}=\sigma(\minustheta(0)\tau^{i+1},0,0),C_{3,i}=\sigma(\minustheta(1)\tau^{i+1},0,0)\}$.
If $\CapitalTheta^i(C_{j,0})=C_{j,i}$ for some $j$, then $\tau^i\in\GF(q)$ and so
$\r^2+\r+1\bigm| i$, contradicting $1\le i<\r^2+\r+1$. 
If $\C=\CapitalTheta^i(\C)$ and  $\CapitalTheta^i(C_{j,0})\neq C_{j,i}$ for $j=1,2,3$,
then  we have either
\begin{eqnarray*}
& \tau=b\minustheta(0)\tau^{i+1},\quad \minustheta(0)\tau=c\minustheta(1)\tau^{i+1},\quad 
\minustheta(1)\tau=d\tau^{i+1},\\
{\rm or\ }& \tau=r\minustheta(1)\tau^{i+1},\quad \minustheta(1)\tau=s\minustheta(0)\tau^{i+1},\quad 
\minustheta(0)\tau=t\tau^{i+1},
\end{eqnarray*} for some nonzero
$b,c,d,r,s,t\in\GF(q)$.
Multiplying the three equations together in each set gives $1=bcd\tau^{3i}$ or $1=rst\tau^{3i}$.  As $\r^2+\r+1=7$ and $(3,7)=1$, this cannot happen for $i<\r^2+\r+1=7$.

We now prove part \ref{sod}.  
By \cite[Lemma~\ref{TS-grouplemma}(\ref{TS-G9})]{barwtgt1}, we can without loss
of generality prove this for the \tangentshadow\ whose coordinates are
calculated in
Corollary~\ref{limage}, namely
$\D=\{(0,\tau,0)\}\cup\{(0,\minustheta(e)^2\tau,0)\st e\in\GF(\r)\}$. 
Using a similar argument to part 1, if $q\geq 3$, there exists $a,b,e,f\in\GF(\r)$ with
$a\ne b$ and $e\ne f$, and so there exists nonzero $c,d\in\GF(q)$ such that
\begin{eqnarray*}
\minustheta(a)^2\tau=c\minustheta(e)^2\tau^{i+1},\quad \minustheta(b)^2\tau=d\minustheta(f)^2\tau^{i+1}.
\end{eqnarray*}
Hence $\minustheta(a)^2/\minustheta(b)^2=(c/d)\minustheta(e)^2/\minustheta(f)^2$.  Thus
$c/d$ is a square in $\GF(\r^3)$, and hence is a square in $\GF(\r)$.  Writing $c/d=x^2$ with $x\in\GF(\r)$ yields
$\minustheta(a)\minustheta(f)=\pm x\minustheta(b)\minustheta(e)$.  
The result now follows using a similar argument to part 
\ref{soc}. The case $q=2$ is similar.
\end{proof}

Consider a 3-space $\Sigma$ of $\PG(6,q)\setminus\si$ about a splash element $[L]$. There are many
special twisted cubics in $\Sigma$, 
each gives rise to a \tangentshadow\ in $[L]$. We now show that there are
only $q^2+q+1$ distinct \tangentshadows\ in $[L]$. 

\begin{lemma}\Label{countlimage}
There are $\r^2+\r+1$ distinct \tangentshadows\ in a non-centre splash element.
\end{lemma}

\begin{proof}
Let $\overline\ell$ be 
a line of $\PG(2,q^3)$ that meets $\li$ in the point $L$. We look at certain
collineations from $\PGL(2,q^3)$ acting on the line $\overline\ell$. In
particular, let $E$ be
the subgroup of $\PGL(2,q^3)$ fixing the infinite point $L$ of
$\overline\ell$ ($E$ is 
defined in Theorem~\ref{PGL2r3}). In $\PG(6,q)$, $\overline\ell$ corresponds to a
3-space $[\overline\ell]$ that meets $\si$ in the spread element $[L]$, and $[E]$ is the subgroup of
$\PGL(4,q)$ acting on $[\overline\ell]\cong\PG(3,q)$ that fixes $\pi_\infty=[L]$.

It is straightforward to show that there are
$\r^3(\r^3-1)$ \orsls\  of $\overline\ell$ disjoint from $L$. 
By Theorem~\ref{PGL2r3}, $E$ is transitive on these \orsls. Hence using
Theorem~\ref{sublinesinBB}, in $\PG(6,q)$,  $[E]$ is transitive on the special twisted cubics of $[\overline\ell]$. However, $|E|=q^3(q^3-1)$, hence $[E]$ acts regularly on the special twisted cubics in $[\overline\ell]$.  

To find the number of distinct \tangentshadows\ in $[L]$, we consider the subgroup $[D]$ of
$[E]$, where $[D]$ contains only the collineations of $\PG(3,\r)$ which fix
$\pi_\infty=[L]$ pointwise.  So $D$ contains the elations and
homologies of $\PG(3,\r)$ with axis $\pi_\infty$. As the product of two
collineations with axis $\pi_\infty$ is also a collineation with axis
$\pi_\infty$, $[D]$ contains exactly the elations and homologies with
axis $\pi_\infty$. The group of elations with axis $\pi_\infty$ and centre a point of $\pi_\infty$ is of size $q$, so the number of non-identity elations of $[D]$ is
$(\r^2+\r+1)\times(\r-1)$. The group of homologies with axis $\pi_\infty$ and centre a point not on $\pi_\infty$ has size $q-1$, hence the number of non-identity homologies of $[D]$ is
$\r^3\times(\r-2)$. 
As the product of two collineations with axis $\pi_\infty$ is again a collineation with axis $\pi_\infty$, 
$|[D]|=(\r^2+\r+1)\times(\r-1)+\r^3\times(\r-2)+1=\r^3(\r-1)$.  Note that
all the special twisted cubics in an orbit of $[D]$ have the
same \tangentshadow, since $[D]$ fixes $\pi_\infty$ pointwise, and hence
fixes the
\tangentshadow\ pointwise.  Hence there are at most $|[E]|/|[D]|=\r^2+\r+1$
distinct \tangentshadows.  However, by Lemma~\ref{singerorbits}, there are at
least $\r^2+\r+1$ distinct \tangentshadows\ and so the result follows.
\end{proof}

We now begin the proof of step 2 of Construction~\ref{construction}. We want to
construct a map from
the \tangentshadow\  of a non-centre splash element $[U]$ to a special conic of
the centre $[T]$. We begin by
investigating a map arising from the cover planes of $\ST$. 

Recall from Theorem~\ref{coverplanes} that a line $m$ of the centre $[T]$ lies in a
unique cover plane which meets any non-centre splash element $[U]$ in a
unique point $M$. Conversely, any  point $M$ of a non-centre splash
element $[U]$ lies in a unique cover plane that meets $[T]$ in a line
$m$. 
We call this bijection from 
points of $[U]$ to lines of $[T]$, and lines of $[T]$ to points of $[U]$ the {\bf cover plane map}.

\begin{theorem}\Label{coverplanestoconic}
In $\PG(6,q)$, let $[\ST]$ be a tangent splash with centre $[T]$ and let $[U]$ be any
non-centre element of $[\ST]$.  
Then the cover plane map takes the $q+1$ lines of $[T]$ through a fixed
point of $[T]$ to a special conic of $[U]$.
Further, the points of a special conic in $[U]$ are mapped under the cover plane map to
lines of $[T]$ through a common point. 
\end{theorem}

\begin{proof}
By Theorem~\ref{transitiveontangentsubplanes}, we can without loss of generality prove the
result for the tangent 
splash with coordinates $\ST=\{(a+b\tau,1,0)\st a,b\in\GF(q)\}\cup\{T=(1,0,0)\}$ (given in 
Section~\ref{coord-subplane}).  Note that by
\cite[Theorem~\ref{TS-nos5-version2-part2}]{barwtgt1}, this tangent splash is
uniquely determined by the centre $T=(1,0,0)$ and the three points 
$U=(0,1,0)$, $V=(1,1,0)$ and $W=(\tau,1,0)$. 
We now work in the Bruck-Bose representation in $\PG(6,q)$ and make use of the
generalised Bruck-Bose map for our coordinates.  Recall that these
coordinates use $\tau^3=t_0+t_1\tau+t_2\tau^2$.

We will use the function $f\colon \GF(\r)\rightarrow \GF(\r^3)$ defined by
$f(a)=1+at_2-a^2t_1+a\tau+a^2\tau^2$.
Consider the $q+1$ lines $m_a$, $a\in\GF(q)\cup\{\infty\}$, of $[T]$
through the point $P_1=\sigma(1,0,0)$, where $m_\infty=\langle \sigma(1,0,0),\sigma(\tau^2,0,0)\rangle$ and 
$m_a=\langle \sigma(1,0,0),\sigma(\tau+a\tau^2,0,0)\rangle$.
Also consider the $q+1$ points of $[U]$ with coordinates
$U_\infty=\sigma(0,-t_1+\tau^2,0)$, $U_a=\sigma(0,f(a),0)$, $a\in\GF(q)$. 
We construct $q+1$ planes $\pi_a=\langle m_a, U_a\rangle$,
$a\in\GF(\r)\cup\{\infty\}$ and 
 show that these planes are cover planes of $\ST$.
As $T,U,V,W$ are not on
a common \orsl\ of $\PG(2,q^3)$, by Theorem~\ref{sublinesinBB}, in
$\PG(6,q)$  $[T],[U],[V],[W]$ are not contained in 
a common 2-regulus. Hence
by Lemma~\ref{splash-coversets-T5}, it suffices to show the planes $\pi_a$ meet
$[T]$ in a line, and $[U],[V],[W]$ in points. 

Clearly the planes $\pi_a$ each meet
$[T]$ in a line and $[U]$ in a point. 
Consider the point $V_a=\sigma(f(a),f(a),0)\in[V]$, $a\in\GF(q)$. We have 
$V_a=(1+at_2-a^2t_1)\sigma(1,0,0)
+a\sigma(\tau+a\tau^2,0,0)+U_a$, so $V_a\in\pi_a$.
Similarly, $V_\infty=\sigma(-t_1+\tau^2,-t_1+\tau^2,0)$ is a point in $[V]$
that also lies in $\pi_\infty$.
Consider the point $W_a=\sigma(f(a)\tau,f(a),0)\in[W]$, $a\in\GF(q)$. Now 
$\tau
f(a)=(1+at_2-a^2t_1+a\tau+a^2\tau^2)\tau=a^2t_0+(1+at_2)(\tau+a\tau^2)$, so
$W_a=a^2t_0\sigma(1,0,0)+(1+at_2)\sigma(\tau+a\tau^2,0,0)+U_a$, so
$W_a\in\pi_a$.
Similarly $W_\infty=\sigma(t_0+t_2\tau^2,-t_1+\tau^2,0)$ is a point in
$[W]$ that is also  in $\pi_\infty$. 
Hence the planes $\pi_a$, $a\in\GF(q)\cup\{\infty\}$, are cover planes of $\ST$. 

We now show that the cover planes $\pi_a$, $a\in\GF(q)\cup\{\infty\}$ meet $[U]$
in $q+1$ points that lie on a special conic. 
The points of $[U]$ have coordinates $(0,0,0,y_0,y_1,y_2,0)$, so we can
consider them as points $(y_0,y_1,y_2)\in\PG(2,q)$. 
In this setting, 
$U_a=(1+at_2-a^2t_1,a,a^2)$ and $U_\infty=(-t_1,0,1)$. It is straightforward
to show that $U_a,U_\infty$ lie on the conic $\C_1$ 
of equation $y_1^2-t_1y_2^2+t_2y_1y_2-y_0y_2=0$. 
To prove that $\C_1$ is special, 
we use Lemma~\ref{transversaleqn} where the transversal point $[U]\cap g$ is
calculated to be $Q=(t_1+t_2\tau-\tau^2,t_2-\tau,-1)$. It is straightforward
to show that $Q$ satisfies the equation of $\C_1$, so $Q\in\C_1^*$. Hence $Q^q,Q^{q^2}\in\C_1^*$ and so $\C_1$ is special.

Thus the pencil of lines of $[T]$ through a point $P_i$ of $[T]$ maps to a special
conic $\C_i$ of $[U]$ via the cover plane map.  If $i\neq j$,
then $\C_i\neq \C_j$ as the cover plane map  maps distinct points of
$[U]$ to distinct lines of $[T]$.
There are $q^2+q+1$ distinct pencils in $[T]$ which map to $q^2+q+1$
distinct special conics of $[U]$. 
By \cite[Lemma 2.6]{barw13} there are exactly $\r^2+\r+1$ special conics in
$[U]$.
Hence the cover plane map acts as a bijection from pencils of $[T]$ to
special conics of $[U]$.
\end{proof}

We can now prove step 2 of Construction~\ref{construction}.

\begin{theorem}\Label{step2}
 There is a bijection from the $q^2+q+1$ \tangentshadows\ $\D$ in a non-centre splash
 element $[U]$ to the $q^2+q+1$
  special conics $\C$ in the centre $[T]$ as follows. Firstly, map the point $D_i\in\D$
  under the cover plane map to a line
  $t_i\in[T]$. Then
\begin{enumerate}
\item if $q$ is even, the lines $t_i$ are concurrent in a point $N$; and there
  is a unique special conic of $[T]$ with nucleus $N$. 
\item if $q$ is odd, the set of lines $t_i$ form a conic envelope, and the
  tangential points form a special conic of $[T]$.
\end{enumerate}
\end{theorem}
\begin{proof}
Suppose first that $q$ is even. By Lemma~\ref{limagelookslike}, the \tangentshadow\ is a
special conic in $[U]$. Hence by
Theorem~\ref{coverplanestoconic}, the cover plane map  maps this to
a set of $q+1$ lines through a point $N$ in $[T]$. In $\PG(6,q^3)$, there is a unique conic
$\C^*$
of $[T]^*$ that contains the three transversal points $Q=g\cap[T]^*$, $Q^q$,
$Q^{q^2}$ and has nucleus $N$. As $N\in[T]$, $\C$ is a conic of
$[T]$. Further, $\C$ is non-degenerate as all special conics are
non-degenerate by \cite[Lemma 2.6]{barw13}.
So we have a map from \tangentshadows\ of $[U]$ to special conics of $[T]$
with distinct \tangentshadows\ mapping to distinct special conics. As
there are $q^2+q+1$ \tangentshadows\ in $[U]$ and $q^2+q+1$ special conics in
$[T]$ (by Lemma~\ref{countlimage} and \cite[Lemma 2.6]{barw13}),
this map is  a bijection.

Now suppose $q$ is odd. 
By Lemma~\ref{limagelookslike}, the \tangentshadow\ is a
set of $q+1$ points in $[U]$, no three in a special conic of $[U]$.
Hence by
Theorem~\ref{coverplanestoconic}, the cover plane map  maps this to
a set of $q+1$ lines in $[T]$, no three concurrent. 
By Segre \cite{segr55} this is a dual conic, and so the tangential points form a
conic of $[T]$. To show that this conic is special, we need to use coordinates.

As in the proof of Theorem~\ref{coverplanestoconic}, without loss of
generality we let $\ST$ be the tangent splash with centre $T=(1,0,0)$ and
containing the three points $U=(0,1,0)$, $V=(1,1,0)$ and $W=(\tau,1,0)$. We now work in
$\PG(6,q)$ using coordinates. 
By
\cite[Lemma~\ref{TS-grouplemma}(\ref{TS-G9})]{barwtgt1}, we can without loss of
generality prove the result for the \tangentshadow\ whose coordinates are
calculated in Corollary~\ref{limage}, namely $
\D=\{D_e=\sigma(0,\minustheta(e)^2\tau,0)\st e\in\GF(\r)\cup\{\infty\}\}$.
We will show that the map described in the statement of the theorem maps
this \tangentshadow\ to the tangent lines of the special conic
$\C=\{C_e=\sigma(\minustheta(e)\tau,0,0)\st e\in\GF(\r)\cup\{\infty\}\}$ in
$[T]$; and so maps it to the conic $\C$. 
By Lemma~\ref{cpoints}, the tangent to $\C$ at the
point $C_e$ is the line $C_eT_e$, with
$T_e=\sigma(\tau\minustheta(e)^2,0,0)$, $e\in\GF(q)$ and
$T_\infty=\sigma(\tau^2,0,0)$.

For $e\in\GF(q)\cup\{\infty\}$, consider the plane $\pi_e=\langle D_e, C_e,T_e\rangle$. 
 It contains a line $C_eT_e$ of
$[T]$ and a point $D_e$ of $[U]$. By Lemma~\ref{splash-coversets-T5}, to
show that $\pi_e$ is a cover plane of the tangent 
splash $[\ST]$, we just need to show that $\pi_e$  meets $[V]$
and $[W]$. 
Suppose first that $e\in\GF(q)$.
Consider the point $V_e=\sigma(\tau\minustheta(e)^2,\tau\minustheta(e)^2,0)$ of
$[V]$, $e\in\GF(q)$, as $V_e=T_e+D_e$, $V_e\in\pi_e$, so $\pi_e$ meets
$[V]$ in a point. 
Consider the point $W_e=\sigma(\tau^2\minustheta(e)^2,\tau\minustheta(e)^2,0)$ of
$[W]$, $e\in\GF(q)$. Now $W_e=\plustheta(e)C_e-e T_e+D_e\in\pi_e$, so
$\pi_e$ meets $[W]$ in a point.
Hence $\pi_e$ is a cover plane of the splash $[\ST]$, $e\in\GF(q)$.
Now let $V_\infty=\sigma(\tau,\tau,0)\in[V]$, and
$W_\infty=\sigma(\tau^2,\tau,0)\in[W]$. As $V_\infty=C_\infty+D_\infty$
and $W_\infty=T_\infty+D_\infty$, we have
$V_\infty,W_\infty\in\pi_\infty$. Hence $\pi_\infty$ is a cover plane of $[\ST]$.

 Thus
the cover plane map maps the point $D_e$
of $\D$ to the tangent  $C_eT_e$ of $\C$, $e\in\GF(q)\cup\{\infty\}$. Hence we have a map from the
\tangentshadow\ $\D$ of $[U]$ to the special conic $\C$ of $[T]$.  
By 
Lemma~\ref{countlimage}, there are $q^2+q+1$ \tangentshadows\ in
$[U]$,
and by 
\cite[Lemma 2.6]{barw13} there are $q^2+q+1$ special conics in $[T]$.
As the cover plane map maps distinct points of $[U]$ to distinct lines of
$[T]$, it follows that distinct \tangentshadows\ are mapped to distinct
special conics. 
Hence our map from $\D$ to $\C$ is a bijection from
\tangentshadows\ of $[U]$ to special conics of $[T]$. 
\end{proof}

This completes the proof of step 2 of Construction~\ref{construction}.

\begin{corollary}\Label{cor-phi} These results give us a map $\phi$ from the special twisted cubic
$\N=[\ell]$ to a special conic $\C$ in $[T]$ by mapping:
a point $N_i\in\N$ to a point $D_i$ in the
\tangentshadow\ $\D$ in $[L]$, then to a line $t_i$ in $[T]$, then to a point $C_i=t_i\cap\C$. 
\end{corollary}

The next theorem shows that this map $\phi$ from $N_i$ to $C_i$, $i=1,\ldots,q+1$, is a projectivity. 
Note that we can easily define a projectivity $\eta$ from $\N$ to $\C$ as follows. Let $\N^*$ have transversal points $P=g\cap[L]^*,\ P^q,P^{q^2}$, and
let $\C^*$ have transversal points $Q=g\cap[T]^*,\ Q^q,Q^{q^2}$. As a
projectivity is uniquely determined by three points, there is a
projectivity $\eta$ from $\N^*$ to $\C^*$ that maps $P\mapsto Q$, $P^q\mapsto
Q^q$, $P^{q^2}\mapsto Q^{q^2}$. The next result shows that this
projectivity $\eta$ also maps
$N_i$ to $C_i$, $i=1,\ldots,q+1$, and so $\eta=\phi$.
Note that the projectivity $\eta$ determines a unique ruled
surface $\vbbb$, so $\vbbb$ is the unique ruled surface  with directrices
$\C,\N$ such that in the cubic extension $\PG(6,q^3)$, $\vbbb^*$ contains the
transversals of the regular spread $\S$. 

\begin{theorem}\Label{nrctoconic}
Let $\N$ be a special twisted cubic in a 3-space of $\PG(6,q)\setminus\si$ about a non-centre
splash
element. Let $\C$ be the special conic in $[T]$
constructed from $\N$ by the map $\phi$, where $\phi$ maps a point
$N_i\in\N$ to a point $C_i\in\C$, $i=1,\ldots,q+1$ as described in Corollary~\ref{cor-phi}. 
Let $\vbbb$ be the unique ruled surface  with directrices
$\C$ and $\N$ such that in the cubic extension $\PG(6,q^3)$, $\vbbb^*$ contains the
transversals of the regular spread $\S$.
Let $\eta$ be the projectivity of $\vbbb$ that maps $\N$ to $\C$. Then
$\eta=\phi$, so $N_iC_i$, $i=1,\ldots,q+1$, are the generator lines of
$\vbbb$. 
\end{theorem}



\begin{proof}
As in the proof of Theorem~\ref{coverplanestoconic}, without loss of generality we may take the splash
$\ST$ to have centre $T=(1,0,0)$ and three points $U=(0,1,0)$, $V=(1,1,0)$ and $W=(\tau,1,0)$.
Further, by 
\cite[Lemma~\ref{TS-grouplemma}(\ref{TS-G7},\ref{TS-G9})]{barwtgt1}
we can without loss of generality take the twisted cubic
$\N=[\ell]$ to correspond to the \orsl\
$\ell=\{(e+\tau,e,e+\tau)\st e\in\GF(q)\}\cup\{(1,1,1)\}$ of $\PG(2,q^3)$.

 Using the notation from Table~\ref{table-coords}, we have
 $\ell=\ell_1=\ell_{e,1,0}$. The line $\ell$ has points 
$P_{e,1}=(e+\tau,e,e+\tau)\equiv(\plustheta(e),e\minustheta(e),\plustheta(e))$,
$e\in\GF(q)$ and $P_\infty=U_0=(1,1,1)$. In
$\PG(6,q)$ label the points of the twisted cubic $\N=[\ell]$ as 
$N_e=P_{e,1}$, $e\in\GF(q)$, and $N_\infty=P_\infty$.
By Corollary~\ref{limage}, the shadow of $N_e$ is the point
$D_e=\sigma(0,\tau\minustheta(e)^2,0)$, $e\in\GF(q)$, and the shadow of
$N_\infty$ is $D_\infty=\sigma(0,\tau,0)$. 

As $\tau$ satisfies the polynomial $x^3-t_2x^2-t_1x-t_0$ we have
$\tau\tau^q\tau^{q^2}=t_0$, $\tau+\tau^q+\tau^{q^2}=t_2$ and
$t_0/\tau=-t_1-t_2\tau+\tau^2$. Hence for $e\in\GF(q)$, 
$N_e=\sigma(\plustheta(e),e\minustheta(e),\plustheta(e))=(\plustheta(e),0,0,e^3+t_2e^2-t_1e,-e^2-t_2e,e,\plustheta(e))$.
Also note that  $\plustheta(e)=e^3+t_2e^2-t_1e+t_0$. 
We can extend
$\N$ to a twisted cubic $\N^*$ of $\PG(6,q^3)$ where $\N^*=\{N_e
\st e\in\GF(q^3)\cup\{\infty\}\}$. 
  We can calculate $N_{-\tau}\equiv(0,0,0,\ t_2\tau+t_1-\tau^2,t_2-\tau,-1,\
  0)$. Hence by Lemma~\ref{transversaleqn}, $N_{-\tau}$ is the transversal point
  $g\cap[U]$.  Similarly
  $N_{-\tau^{\r}}=g^{\r}\cap [U]$ and $N_{-\tau^{\r^2}}=g^{\r^2}\cap
  [U]$.

The proof of Theorem~\ref{step2} shows that the bijection $\phi$ maps $\N$
to the conic
$\C=\{C_e=\sigma(\minustheta(e)\tau,0,0)\st e\in\GF(q)\cup\{\infty\}\}$ in $[T]$
with $\phi(N_e)=C_e$, $e\in\GF(q)\cup\{\infty\}$. 
Note that the proof of Lemma~\ref{cpoints} shows that $\C$ is a special
conic, and that when we expand $\sigma$, we have
$C_e=(t_0,e^2+t_2e,-e,\ 0,0,0, \ 0)$,
$e\in\GF(\r)\cup\{\infty\}$. We can extend $\C$ to a conic of
$\PG(6,q^3)$, by taking
$e\in\GF(\r^3)\cup\{\infty\}$.  Now $C_{-\tau}=(t_0,\tau^2-t_2\tau,\tau,\ 0,0,0,\ 0)\equiv(t_2\tau+t_1-\tau^2,t_2-\tau,-1,\ 0,0,0,\ 0)$ which is the transversal point
  $g\cap[T]$.  Similarly, $C_{-\tau^\r}=g^\r\cap[T]$ and $C_{-\tau^{\r^2}}=g^{\r^2}\cap[T]$.

We now show that the mapping $\phi\colon N_e\mapsto C_e$,
$e\in\GF(q)\cup\{\infty\}$, is a projectivity.  Firstly, note that the map
$(1,e,e^2,e^3)\mapsto (1,e,e^2)$ for $e\in\GF(\r)\cup\{\infty\}$ is a
projectivity.  Now the point $C_e$ is equivalent to $M_C(1,e,e^2)^t$ and
the point $N_e$ is equivalent to $M_N(1,e,e^2,e^3)^t$ where
\[
M_C=\begin{pmatrix}
t_0&0&0\\0&t_2&1\\0&-1&0
\end{pmatrix},
\quad
\quad
M_N=\begin{pmatrix}
0&-t_1&t_2&1\\
0&-t_2&-1&0\\
0&1&0&0\\
t_0&-t_1&t_2&1
\end{pmatrix}.
\] Hence the map $\phi$ is a projectivity.  Extend $\phi$ in the natural way to a
projectivity  of $\PG(6,q^3)$. It is straightforward to verify that this
projectivity maps $N_{-\tau},N_{-\tau^\r},N_{-\tau^{q^2}}$ to
$C_{-\tau},C_{-\tau^\r},C_{-\tau^{q^2}}$ respectively. Thus the projectivity $\phi$ corresponds to the
unique ruled surface with directrices $\C$ and $\N$ containing the
transversals of $\S$, that is $\phi=\eta$.
\end{proof}

We now complete the proof of Construction~\ref{construction}.

{\bf Proof of Construction~\ref{construction}\ } 
In $\PG(2,q^3)$, let $\ST$ be a tangent splash of $\li$ and let $\ell$ be a
fixed \orsl\ such that $\overline\ell$ meets $\li$ in a point of $\ST\setminus\{T\}$.
By Theorem~\ref{nrctoconic}, Steps 1, 2, 3 of Construction~\ref{construction}
give a ruled surface $\vbbb$ of $\PG(6,q)$ that corresponds to an \orsp\ $\pi$
of $\PG(2,q^3)$. As $\vbbb$ contains the twisted cubic $\N=[\ell]$, $\pi$
contains the \orsl\ $\ell$. It remains to show that $\pi$ has tangent splash $\ST$.

By Theorem~\ref{splashextsubplane}, in $\PG(2,q^3)$, there is a
unique \orsp\ that has centre $T$, tangent splash $\ST$ and contains
$\ell$, denote this unique subplane by $\Bpi$. 
We will use coordinates to show that $\pi=\Bpi$, by showing that in
$\PG(6,q)$, $\vbbb=[\Bpi]$.

As before, without loss of generality let
$\ST=\{(a+b\tau,1,0)\st a,b\in\GF(q)\}\cup\{T=(1,0,0)\}$ and $\ell=\ell_1$. So
the unique \orsp\ containing $\ell$ with splash $\ST$ is the \orsp\ $\Bpi$
coordinatised in Section~\ref{coord-subplane}.
In Lemma~\ref{cpoints}, the coordinates of the conic directrix of
$[\Bpi]$ in $\PG(6,q)$ are calculated. The proof of Theorem~\ref{nrctoconic} shows that $\vbbb$
has the same conic directrix. So $\vbbb$ and $\bbb$ are ruled surfaces with the same conic directrix
and same twisted cubic directrix. 
By 
Theorem~\ref{FFA-subplane-tangent}, 
in the cubic extension $\PG(6,q^3)$, $[\Bpi]^*$ contains the three transversals
of the regular spread $\S$. 
By Theorem~\ref{nrctoconic}, $\vbbb^*$ also contains the three transversals
of $\S$. 
Since a projectivity is uniquely determined by the image of three points,
both $\vbbb$ and $[\Bpi]$ are determined by the same projectivity, and so 
 $\vbbb=[\Bpi]$. Hence in $\PG(2,q^3)$, 
$\Bpi$ and $\pi$ are the same \orsp, so the \orsp\ $\pi$ constructed in
Construction~\ref{construction} is the unique \orsp\ that has
tangent splash $\ST$ and contains the \orsl\ $\ell$. 
\hfill$\square$

This completes the proof of Construction~\ref{construction}.

\section{Conclusion}

This paper concludes the study of tangent \orsps\ in $\PG(2,q^3)$ in the Bruck-Bose representation in $\PG(6,q)$. In \cite{barw12}, we investigated the \orsls\ and \orsps\ of $\PG(2,q^3)$, and in particular showed that the tangent \orsps\ correspond to certain ruled surfaces in $\PG(6,q)$. In \cite{barw13}, we characterised which ruled surfaces of $\PG(6,q)$ correspond to tangent \orsps. The main results of \cite{barwtgt1} involved finding properties of the tangent splash of a tangent \orsp. Further, we investigated the tangent space of a point of the ruled surface in $\PG(6,q)$ representing a tangent \orsp. In this paper, building on earlier results, we investigated properties of the tangent splash in the $\PG(6,q)$ setting. Further, we gave a geometrical construction of a tangent \orsp\ in $\PG(6,q)$, beginning with a tangent splash and one appropriate \orsl.

\end{document}

%% file: tangentsubplane3.pdftex_t
\begin{picture}(0,0)%
\includegraphics{tangentsubplane3.pdftex}%
\end{picture}%
\setlength{\unitlength}{4144sp}%
\begingroup\makeatletter\ifx\SetFigFont\undefined%
\gdef\SetFigFont#1#2#3#4#5{%
  \reset@font\fontsize{#1}{#2pt}%
  \fontfamily{#3}\fontseries{#4}\fontshape{#5}%
  \selectfont}%
\fi\endgroup%
\begin{picture}(4732,2909)(-1456,-4373)
\put(1756,-3561){\makebox(0,0)[lb]{\smash{{\SetFigFont{12}{14.4}{\rmdefault}{\mddefault}{\updefault}{\color[rgb]{0,0,0}$m_h$}%
}}}}
\put(448,-1635){\makebox(0,0)[lb]{\smash{{\SetFigFont{12}{14.4}{\familydefault}{\mddefault}{\updefault}{\color[rgb]{0,0,0}$T(1,0,0)$}%
}}}}
\put(3261,-1786){\makebox(0,0)[lb]{\smash{{\SetFigFont{12}{14.4}{\rmdefault}{\mddefault}{\updefault}{\color[rgb]{0,0,0}$\li$}%
}}}}
\put(1426,-2626){\makebox(0,0)[lb]{\smash{{\SetFigFont{12}{14.4}{\rmdefault}{\mddefault}{\updefault}{\color[rgb]{0,0,0}$R_{d,e,f,h}$}%
}}}}
\put(2626,-2076){\makebox(0,0)[lb]{\smash{{\SetFigFont{12}{14.4}{\rmdefault}{\mddefault}{\updefault}{\color[rgb]{0,0,0}$\ell_{e,d,f}$}%
}}}}
\put(891,-2966){\makebox(0,0)[lb]{\smash{{\SetFigFont{12}{14.4}{\rmdefault}{\mddefault}{\updefault}{\color[rgb]{0,0,0}$P_{\!e,d}$}%
}}}}
\put(231,-3761){\makebox(0,0)[lb]{\smash{{\SetFigFont{12}{14.4}{\rmdefault}{\mddefault}{\updefault}{\color[rgb]{0,0,0}$m_0$}%
}}}}
\put(1071,-4116){\makebox(0,0)[lb]{\smash{{\SetFigFont{12}{14.4}{\rmdefault}{\mddefault}{\updefault}{\color[rgb]{0,0,0}$m_e$}%
}}}}
\put(-334,-3481){\makebox(0,0)[lb]{\smash{{\SetFigFont{12}{14.4}{\rmdefault}{\mddefault}{\updefault}{\color[rgb]{0,0,0}$m_\infty$}%
}}}}
\put(-549,-1951){\makebox(0,0)[lb]{\smash{{\SetFigFont{12}{14.4}{\rmdefault}{\mddefault}{\updefault}{\color[rgb]{0,0,0}$\ell_d$}%
}}}}
\put(554,-2539){\makebox(0,0)[lb]{\smash{{\SetFigFont{12}{14.4}{\rmdefault}{\mddefault}{\updefault}{\color[rgb]{0,0,0}$S_d$}%
}}}}
\put(-49,-2484){\makebox(0,0)[lb]{\smash{{\SetFigFont{12}{14.4}{\familydefault}{\mddefault}{\updefault}{\color[rgb]{0,0,0}$P_{\!\infty}$}%
}}}}
\put(-446,-3176){\makebox(0,0)[lb]{\smash{{\SetFigFont{12}{14.4}{\rmdefault}{\mddefault}{\updefault}{\color[rgb]{0,0,0}$U_{\!f}$}%
}}}}
\put(2756,-2996){\makebox(0,0)[lb]{\smash{{\SetFigFont{12}{14.4}{\rmdefault}{\mddefault}{\updefault}{\color[rgb]{0,0,0}${\mathscr B}$}%
}}}}
\end{picture}%

%% file: define-the-I.pdftex_t
\begin{picture}(0,0)%
\includegraphics{define-the-I.pdftex}%
\end{picture}%
\setlength{\unitlength}{4144sp}%
\begingroup\makeatletter\ifx\SetFigFont\undefined%
\gdef\SetFigFont#1#2#3#4#5{%
  \reset@font\fontsize{#1}{#2pt}%
  \fontfamily{#3}\fontseries{#4}\fontshape{#5}%
  \selectfont}%
\fi\endgroup%
\begin{picture}(7020,4044)(1334,-3584)
\put(3871,-306){\makebox(0,0)[lb]{\smash{{\SetFigFont{12}{14.4}{\rmdefault}{\mddefault}{\updefault}{\color[rgb]{0,0,0}$\overline\ell_1$}%
}}}}
\put(6358, 89){\makebox(0,0)[lb]{\smash{{\SetFigFont{12}{14.4}{\rmdefault}{\mddefault}{\updefault}{\color[rgb]{0,0,0}$I_m$}%
}}}}
\put(6578,-538){\makebox(0,0)[lb]{\smash{{\SetFigFont{12}{14.4}{\rmdefault}{\mddefault}{\updefault}{\color[rgb]{0,0,0}$I_{P_1}$}%
}}}}
\put(5856,-2393){\makebox(0,0)[lb]{\smash{{\SetFigFont{12}{14.4}{\rmdefault}{\mddefault}{\updefault}{\color[rgb]{0,0,0}$[\ell_1]$}%
}}}}
\put(6463,-175){\makebox(0,0)[lb]{\smash{{\SetFigFont{12}{14.4}{\rmdefault}{\mddefault}{\updefault}{\color[rgb]{0,0,0}$I_{P_q}$}%
}}}}
\put(5926,-3511){\makebox(0,0)[lb]{\smash{{\SetFigFont{12}{14.4}{\rmdefault}{\mddefault}{\updefault}{\color[rgb]{0,0,0}$\PG(6,q)$}%
}}}}
\put(5531,-78){\makebox(0,0)[lb]{\smash{{\SetFigFont{12}{14.4}{\rmdefault}{\mddefault}{\updefault}{\color[rgb]{0,0,0}$[T]$}%
}}}}
\put(5586,-730){\makebox(0,0)[lb]{\smash{{\SetFigFont{12}{14.4}{\rmdefault}{\mddefault}{\updefault}{\color[rgb]{0,0,0}$\mathcal C$}%
}}}}
\put(5306,204){\makebox(0,0)[lb]{\smash{{\SetFigFont{12}{14.4}{\rmdefault}{\mddefault}{\updefault}{\color[rgb]{0,0,0}$\si$}%
}}}}
\put(4493,-1265){\makebox(0,0)[lb]{\smash{{\SetFigFont{12}{14.4}{\rmdefault}{\mddefault}{\updefault}{\color[rgb]{0,0,0}$\longleftrightarrow$}%
}}}}
\put(3878,169){\makebox(0,0)[lb]{\smash{{\SetFigFont{12}{14.4}{\rmdefault}{\mddefault}{\updefault}{\color[rgb]{0,0,0}$L_1$}%
}}}}
\put(3071,-1681){\makebox(0,0)[lb]{\smash{{\SetFigFont{12}{14.4}{\rmdefault}{\mddefault}{\updefault}{\color[rgb]{0,0,0}$\ell_q$}%
}}}}
\put(2231,-3406){\makebox(0,0)[lb]{\smash{{\SetFigFont{12}{14.4}{\rmdefault}{\mddefault}{\updefault}{\color[rgb]{0,0,0}$\PG(2,q^3)$}%
}}}}
\put(2106,-1491){\makebox(0,0)[lb]{\smash{{\SetFigFont{12}{14.4}{\rmdefault}{\mddefault}{\updefault}{\color[rgb]{0,0,0}$P_1$}%
}}}}
\put(2376,-496){\makebox(0,0)[lb]{\smash{{\SetFigFont{12}{14.4}{\rmdefault}{\mddefault}{\updefault}{\color[rgb]{0,0,0}$P_q$}%
}}}}
\put(2276,-1741){\makebox(0,0)[lb]{\smash{{\SetFigFont{12}{14.4}{\rmdefault}{\mddefault}{\updefault}{\color[rgb]{0,0,0}$m$}%
}}}}
\put(3143,-908){\makebox(0,0)[lb]{\smash{{\SetFigFont{12}{14.4}{\rmdefault}{\mddefault}{\updefault}{\color[rgb]{0,0,0}$\ell_1$}%
}}}}
\put(3186,-2331){\makebox(0,0)[lb]{\smash{{\SetFigFont{12}{14.4}{\rmdefault}{\mddefault}{\updefault}{\color[rgb]{0,0,0}$\pi$}%
}}}}
\put(2776,144){\makebox(0,0)[lb]{\smash{{\SetFigFont{12}{14.4}{\rmdefault}{\mddefault}{\updefault}{\color[rgb]{0,0,0}$T$}%
}}}}
\put(1645,126){\makebox(0,0)[lb]{\smash{{\SetFigFont{12}{14.4}{\rmdefault}{\mddefault}{\updefault}{\color[rgb]{0,0,0}$\li$}%
}}}}
\put(6105,-1835){\makebox(0,0)[lb]{\smash{{\SetFigFont{12}{14.4}{\rmdefault}{\mddefault}{\updefault}{\color[rgb]{0,0,0}$P_1$}%
}}}}
\put(7253,-740){\makebox(0,0)[lb]{\smash{{\SetFigFont{12}{14.4}{\rmdefault}{\mddefault}{\updefault}{\color[rgb]{0,0,0}$I_{P_1,\ell_1}$}%
}}}}
\put(7748,-48){\makebox(0,0)[lb]{\smash{{\SetFigFont{12}{14.4}{\rmdefault}{\mddefault}{\updefault}{\color[rgb]{0,0,0}$[L_1]$}%
}}}}
\end{picture}%

%% file: tgt2-splash-paper.bbl
\begin{thebibliography}{999}

\bibitem{andr54} J.\ Andr\'e. \"Uber nicht-Desarguessche Ebenen mit
 transitiver Translationgruppe. {\em Math.\ Z.}, {\bf 60} (1954)
 156--186.


\bibitem{barw12} S.G. Barwick and W.A. Jackson. Sublines and subplanes of
  $\PG(2,q^3)$ in  the Bruck-Bose
  representation in $\PG(6,q)$. Finite Fields Th.  App. 18 (2012) 93--107.

\bibitem{barw13} S.G. Barwick and W.A. Jackson. A characterisation of
  tangent subplanes of $\PG(2,\r^3)$. 
To appear, Designs Codes, Crypt. (DOI) 10.1007/s10623-012-9754-7.


\bibitem{barwtgt1} S.G. Barwick and W.A. Jackson. An investigation of the
  tangent splash of a subplane of $\PG(2,q^3)$. Submitted. http://arxiv.org/abs/1303.5509
  
\bibitem{bruc69} R.H.\ Bruck.  Construction problems of finite
projective planes.  {\em Conference on Combinatorial Mathematics
and its Applications},  University of North Carolina Press, (1969)
426--514.


\bibitem{bruc64} R.H.\ Bruck and R.C.\ Bose. The construction of
 translation planes from projective spaces. {\em J.\ Algebra}, {\bf
 1} (1964) 85--102.

\bibitem{bruc66} R.H.\ Bruck and R.C.\ Bose. Linear
representations of projective planes in projective spaces. {\em
J.\ Algebra}, {\bf 4} (1966) 117--172.

\bibitem{rey} L.R.A.\ Casse.
{\em A Guide to Projective Geometry.} Oxford University Press, 2006.

\bibitem{hirs85} J.W.P.\ Hirschfeld. {\em Finite Projective Spaces of
 Three Dimensions.} Oxford University Press, 1985.

\bibitem{hirs98} J.W.P. Hirschfeld. {\em Projective Geometry over Finite
Fields, Second Edition.} Oxford University Press, 1998.

\bibitem{hirs91} J.W.P.\ Hirschfeld and J.A.\ Thas. {\em General
  Galois Geometries.} Oxford University
  Press, 1991.


\bibitem{segr55} B. Segre. Ovals in a finite projective plane, {\em Canad.
J. Math.} {\bf 7} (1955): 414--416. 
\end{thebibliography}
